\documentclass[a4paper,10pt,leqno]{article}

\usepackage{amsmath,amssymb}
\usepackage{amsthm}
\usepackage{mathrsfs}
\usepackage[all]{xy}
\usepackage{rotate}
\usepackage{enumerate}

\newcommand{\Lie}{\textup{Lie}}

\newcommand{\Ad}{\textup{Ad}}
\newcommand{\ad}{\textup{ad}}

\newcommand{\GL}{\textup{GL}}
\newcommand{\gl}{\mathfrak{gl}}

\newcommand{\slor}{\mathfrak{sl}}

\newcommand{\Sp}{\textup{Sp}}
\newcommand{\spor}{\mathfrak{sp}}

\newcommand{\upO}{\textup{O}}
\newcommand{\SO}{\textup{SO}}
\newcommand{\so}{\mathfrak{so}}

\newcommand{\su}{\mathfrak{su}}
\newcommand{\uor}{\mathfrak{u}}

\newcommand{\RR}{\mathbb{R}}

\newcommand{\CC}{\mathbb{C}}
\newcommand{\ZZ}{\mathbb{Z}}

\newcommand{\Ind}{\textup{Ind}}
\newcommand{\sgn}{\textup{sgn}}

\newcommand{\calO}{\mathcal{O}}

\newcommand{\calW}{\mathcal{W}}
\newcommand{\calB}{\mathcal{B}}
\newcommand{\calH}{\mathcal{H}}

\newcommand{\calU}{\mathcal{U}}

\newcommand{\calV}{\mathcal{V}}

\newcommand{\calJ}{\mathcal{J}}

\newcommand{\frakg}{\mathfrak{g}}

\newcommand{\frakk}{\mathfrak{k}}
\newcommand{\frakp}{\mathfrak{p}}
\newcommand{\fraks}{\mathfrak{s}}
\newcommand{\frakn}{\mathfrak{n}}
\newcommand{\fraka}{\mathfrak{a}}

\newcommand{\frakm}{\mathfrak{m}}
\newcommand{\frakl}{\mathfrak{l}}
\newcommand{\frakt}{\mathfrak{t}}
\newcommand{\frako}{\mathfrak{o}}
\newcommand{\frakb}{\mathfrak{b}}
\newcommand{\frakh}{\mathfrak{h}}

\newcommand{\frake}{\mathfrak{e}}

\newcommand{\rank}{\textup{rank}}

\newcommand{\Ann}{\textup{Ann}}

\newcommand{\gr}{\textup{gr}}
\newcommand{\Hom}{\textup{Hom}}

\newcommand{\ev}{\textup{ev}}

\theoremstyle{plain}
\newtheorem{theorem}{Theorem}[section]

\newtheorem{proposition}[theorem]{Proposition}
\newtheorem{lemma}[theorem]{Lemma}
\newtheorem{corollary}[theorem]{Corollary}

\newtheorem{fact}[theorem]{Fact}
\newtheorem*{question}{Question}

\theoremstyle{definition}
\newtheorem{definition}[theorem]{Definition}
\newtheorem{example}[theorem]{Example}

\newtheorem{remark}[theorem]{Remark}

\theoremstyle{remark}
\renewenvironment{pf}{\begin{proof}[\sc Proof]}{\end{proof}}

\numberwithin{equation}{section}

\newcommand{\rring}[1]{\mathbb{C}[#1]}
\newcommand{\define}[1]{\textit{#1}}
\newcommand{\dweight}{\Lambda^{+}}
\newcommand{\ch}{\textup{ch}}

\newcommand{\AV}{\mathcal{AV}}
\newcommand{\sor}{\gamma}

\newcommand{\U}{\textup{U}}
\newcommand{\calA}{\mathcal{A}}
\newcommand{\spn}[1]{\mathrm{span}_{\CC}\{#1\}}
\newcommand{\hsum}{\mathop{\sum\nolimits^{\oplus}}\limits}

\newenvironment{texteqn}
{\begin{equation} 
\addtolength{\abovedisplayskip}{-1ex}
\addtolength{\abovedisplayshortskip}{-1ex}
\addtolength{\belowdisplayskip}{-1ex}
\addtolength{\belowdisplayshortskip}{-1ex}
\begin{minipage}[t]{0.87\linewidth}}
{\end{minipage} \end{equation} \ignorespacesafterend}

\begin{document}


\title{Stability of Branching Laws for Highest Weight Modules}
\author{Masatoshi Kitagawa\thanks{Email: kitamasa@ms.u-tokyo.ac.jp} \\
\textit{Graduate School of Mathematical Sciences, the University of Tokyo,} \\
\textit{3-8-1 Komaba Meguro-ku Tokyo 153-8914, Japan}}
\date{}
\maketitle

\begin{abstract}
We say a representation $V$ of a group $G$ has stability if its multiplicities $m^{G}_{V}(\lambda)$ is dependent only on some equivalence class of $\lambda$
for a sufficiently large parameter $\lambda$.

In this paper, we prove that the restriction of a holomorphic discrete series representation with respect to any holomorphic symmetric pairs has stability.
As a corollary, we give a necessary and sufficient condition on multiplicity-freeness of the branching laws in this setting.
This condition is same as the sufficient condition given by the theory of visible actions.

We prove a general theorem before we show the stability of holomorphic discrete series representations.
Using the general theorem, we also show the stability on quasi-affine spherical homogeneous spaces and the stability of $K$-type of unitary highest weight modules.

We also show that two branching laws of a holomorphic discrete series representation coincide if two subgroups are in same $\epsilon$-family.



\end{abstract}

\section{Introduction}

Let $G$ be a simple Lie group of Hermitian type with finite center, $\theta$ be a Cartan involution of $G$.
Put $K=G^{\theta}$.
Suppose $\frakg = \frakk \oplus \frakp$ is the Cartan decomposition determined by $\theta$.
Let $\tau$ be an involutive automorphism of $G$ commuting with $\theta$ that fixes any elements of the center of $\frakk$,
and $H$ be the identity component of the fixed point subgroup $G^{\tau}$.
The pair $(G, H)$ is called a holomorphic symmetric pair.

Suppose $\calH$ is a holomorphic discrete series representation of $G$.
The purpose of this paper is to describe the behavior of the branching law of $\calH|_{H}$.
In \cite[Theorem 7.4]{Ko98},
it was shown that $\calH|_{H}$ is decomposed into the direct sum of holomorphic discrete series representations of $H$
with finite multiplicities.
By this result, we can decompose $\calH|_{H}$ as 
\begin{align*}
\calH|_{H} \simeq \hsum_{\lambda \in \sqrt{-1}(\frakt^{\tau})^{*}} m^{H}_{\calH}(\lambda)V^{H}_{\lambda},
\end{align*}
where $\frakt^{\tau}$ is a Cartan subalgebra of $\frakk^{\tau}$
and $V^{H}_{\lambda}$ is a holomorphic discrete series representation of $H$ with highest weight $\lambda$.
In \cite[Theorem 18, Theorem 38]{Ko05} and \cite{Ko08},
T. Kobayashi showed that the multiplicity function $m^{H}_{\calH}(\lambda)$ is uniformly bounded with respect to $\lambda$,
and gave the sufficient condition for multiplicity-freeness of $\calH|_{H}$ using the theory of visible actions.

The main theorem of this paper says that the sufficient condition on multiplicity-freeness given by visible actions
is also a necessary condition for holomorphic discrete series representations and holomorphic symmetric pairs.
More precisely, the branching laws have a good property called stability.
Our main theorem is as follows (see Theorem \ref{theorem:stable_theorem_for_holomorphic_discrete_series}).
\begin{theorem}\label{theorem:intro_stability_for_discrete}
Suppose $\fraka \subset \frakp^{-\tau}$ is a maximal abelian subspace
that is determined by a positive root system of $(\frakg_{\CC}^{\theta\tau}, \frakt_{\CC}^{\tau})$.
Then, there exists a $\lambda_0 \in \sqrt{-1}(\frakt^{\tau})^{*}$ such that
\begin{align*}
m^{H}_{\calH}(\lambda+\lambda_0)=m^{Z_{K\cap H}(\fraka)}_{\calH^{\frakp_{+}}}(\lambda|_{Z_{T^{\tau}}(\fraka)})
\end{align*}
for any $\lambda \in \sqrt{-1}(\frakt^{\tau})^{*}$ satisfying $m^{H}_{\calH}(\lambda)\neq 0$.
\end{theorem}
This theorem asserts two things:
the multiplicity function $m^{H}_{\calH}(\lambda)$ is periodic for sufficiently large parameter $\lambda$,
and the multiplicities in sufficiently large parameters can be described by the decomposition of $\calH^{\frakp_{+}}$ with respect to $Z_{K\cap H}(\fraka)$.
The first phenomenon is called stability.

Stability was appeared in \cite[Lemma 3.4]{Ko94} for example.
In \cite{Sa93}, F. Sat{\=o} formulated and generalized it for reductive spherical homogeneous spaces.
To prove Theorem \ref{theorem:intro_stability_for_discrete}, we generalize Sat{\=o}'s stability theorem for multiplicity-free spaces.

We will state the stability theorem for multiplicity-free spaces.
Let $G$ be a connected reductive complex algebraic group.
Fix a Borel subgroup $B$ of $G$.
Let $X$ be an irreducible quasi-projective $G$-variety satisfying the following conditions:
\begin{texteqn}
$X$ is a spherical $G$-variety (i.e. a Borel subgroup of $G$ has an open dense orbit in $X$), and
\end{texteqn}
\begin{texteqn}
the quotient field of $\rring{X}$ is naturally isomorphic to the rational function field of $X$.
\end{texteqn}
Note that we do not assume normality for spherical varieties in this paper.
By definition, there exists a point $x_0 \in X$ such that $Bx_0$ is open dense in $X$.
The details of the following theorem is in Section \ref{section:stable_multiplicity}.
\begin{theorem}\label{theorem:intro_stability_for_general}
Let $M$ be a finitely generated $(\rring{X}, G)$-module (see Definition \ref{definition:(A,G)-module}).
Suppose $\rring{X}$ has no zero divisors in $M$:
\begin{align*}
\bigcup_{m \in M \setminus \{0\}}\Ann_{\rring{X}}(m)=0.
\end{align*}
Then, there exists a reductive subgroup $L \subset G_{x_0}$ and $\lambda_0 \in \dweight(\rring{X})$ such that
\begin{align*}
m^G_M(\lambda + \lambda_0)=m^L_{M/\frakm(x_0)M}(\lambda|_{B_{x_0}})
\end{align*}
for any $\lambda \in \dweight(M)$.
\end{theorem}
Here, we denote by $\dweight(V)$ the set of the highest weights of the irreducible representations with respect to $B$
which appear in the irreducible decomposition of $V$,
denote by $G_{x_0}$ the stabilizer at $x_0$ in $G$,
and denote by $\frakm(x_0)$ the maximal ideal of $\rring{X}$ corresponding to $x_0$.
The proof of Sat{\=o}'s stability theorem is based on duality settings such as Schur-Weyl duality and Peter-Weyl theorem.
Our proof is based on a simple observation: images of $B$-eigenvectors by the evaluation map
\begin{align*}
\ev_{x_0}: M \rightarrow M/\frakm(x_0)M
\end{align*}
are also $B_{x_0}$-eigenvectors.

The subgroup $L$ in Theorem \ref{theorem:intro_stability_for_general} can be represented as
\begin{align*}
L=\{g \in G_{x_0}: gB{x_0} \subset B{x_0}\}.
\end{align*}
$L$ is the unique subgroup of $G$ that contains $B_{x_0}$ as a 'Borel subgroup' (see Proposition \ref{proposition:L-subgroup}).
The reductivity and other properties of $L$ was studied by Brion, Luna and Vust in \cite{BLV86}.
For some concrete settings, we can determine the explicit form of $L$.
In the setting of Theorem \ref{theorem:intro_stability_for_discrete}, $L$ is the complexification of $Z_{K\cap H}(\fraka)$.

We apply Theorem \ref{theorem:intro_stability_for_general} to the following three cases:
\begin{enumerate}[i)]
	\item $X=G/H$ and $M=\Ind_H^{G}(W)$ for a quasi-affine spherical homogeneous space and a finite dimensional representation $W$ of $H$.
	\item $X=\AV(\calH)$, $M=\calH$ and $G=K_{\CC}$ for a unitary highest weight module $\calH$.
	\item $X=\frakp_{+}^{-\tau}$, $M=\calH/\frakp_{-}^{-\tau}\calH$ and $G=(H\cap K)_{\CC}$ for a holomorphic discrete series representation $\calH$.
\end{enumerate}
The first case corresponds to Sat{\=o}'s theorem.
These examples are dealt in Section \ref{section:examples}.

As a corollary of Theorem \ref{theorem:intro_stability_for_general}, we can obtain a necessary and sufficient condition for
multiplicity-freeness of $M$:
$M$ is multiplicity-free as a representation of $G$ if and only if $M/\frakm(x_0)M$ is multiplicity-free as a representation of $L$.
This result can be considered as an analogue of a propagation theorem of multiplicity-freeness in the theory of visible actions.
The concept of visible actions first appeared in \cite{Ko04}, and the propagation theorem was proved in \cite{Ko13}.
For some spherical $G$-varieties, it was shown that a compact real form of $G$ acts on them strongly visibly.
(see e.g. \cite{Ko07_2}, \cite{Sa09} and \cite{Ta12})

In Section \ref{section:upper_bound}, we treat similarity of branching laws of holomorphic discrete series representations with respect
to two holomorphic symmetric pairs.
Let $(\frakg, \frakh)$ be a holomorphic symmetric pair.
Suppose $(\frakg, \frakh_\epsilon)$ is an element of the $\epsilon$-family of $(\frakg, \frakh)$,
and $(\frakg, \frakh_\epsilon)$ is a holomorphic symmetric pair.
$\epsilon$-family of a symmetric pair is defined by T. {\=O}shima and J. Sekiguchi in \cite{OsSe80, OsSe84}.
For example, if $(\frakg, \frakh) = (\spor(n,\RR), \uor(n,\RR))$, its $\epsilon$-family is
\begin{align*}
\{(\spor(n,\RR), \uor(n-i,i)): 0 \leq i \leq n\} \cup \{(\spor(n, \RR), \gl(n, \RR))\}.
\end{align*}
Let $H$ and $H_{\epsilon}$ be analytic subgroups of $G$ with their Lie algebra $\frakh$ and $\frakh_\epsilon$.
We note that the complexifications of $\frakh$ and $\frakh_\epsilon$ are conjugate by inner automorphism of the complexification of $\frakg$.
Our concern is the similarity between two branching laws of $\calH|_H$ and $\calH|_{H_\epsilon}$ for a holomorphic discrete series representation $\calH$.
This is motivated by the fact that the theta correspondence of infinitesimal characters is independent of
any choices of real forms of a dual pair, which is due to R. Howe \cite{Ho89}, T. Przebinda \cite{Pr96} and J-S Li \cite{Li99},
and also motivated by Weyl's unitary trick.
The main theorem in Section \ref{section:upper_bound} is as follows.
\begin{theorem}
Suppose $\calH$ is a holomorphic discrete series representation of $G$.
Then, for sufficiently large parameter $\lambda$, we have
\begin{align*}
m_{\calH}^H(\lambda)=m_{\calH}^{H_\epsilon}(\lambda).
\end{align*}
\end{theorem}

\subsection*{Acknowledgement}
The author would like to thank his adviser Prof. Toshiyuki Kobayashi for many helpful advices.

\section{Preliminaries: Some algebraic results}

In this section, we set up some notations and results of representations of algebraic groups.
For a Lie group $G$, we write its Lie algebra by a German letter as $\frakg := \Lie(G)$,
and we write its complexification by a subscript $(\cdot)_\CC$ as $\frakg_{\CC} := \frakg \otimes_{\RR} \CC$.

\subsection{$G$-algebra and $(\calA, G)$-module}
Let $G$ be a connected reductive complex algebraic group.
Fix a Borel subgroup $B$ of $G$.
Let $B=TN$ be its Levi decomposition,
where $T$ is a maximal torus of $G$ and $N$ is the unipotent radical of $B$.
Let $\dweight=\dweight_{G} \subset \frakt^{*}$ be the set of dominant integral weights with respect to $B$.
For each $\lambda \in \dweight$,
we denote by $V_{\lambda}=V_{\lambda, G}$ the irreducible representation of $G$
with highest weight $\lambda$.

For any algebraic group $H$, we say a representation $V$ of $H$ over $\CC$ is a \define{rational representation}
if $\spn{gv: g \in H}$ is a finite dimensional and algebraic representation of $H$ for any $v \in V$.
This implies that any rational representation of $G$ is completely reducible.
Given a rational representation $V$ of $G$, we can decompose $V$ into the direct sum of irreducible representations:
\begin{equation*}
V = \bigoplus_{\lambda \in \dweight}m^{G}_V(\lambda)V_{\lambda}.
\end{equation*}
If the group $G$ is obvious, we write $m_V(\lambda) := m^{G}_V(\lambda)$.
We set
\begin{align*}
\dweight(V):=\dweight_{G}(V):=\{\lambda \in \dweight: m_V(\lambda) \neq 0\}.
\end{align*}

We say a $\CC$-algebra $\calA$ is \define{$G$-algebra}
if $\calA$ is a rational representation of $G$ and $G$ acts on $\calA$ by $\CC$-algebra automorphisms.
\begin{definition}\label{definition:(A,G)-module}
Let $\calA$ be a $G$-algebra, and $M$ be an $\calA$-module and rational representation of $G$.
Then, $M$ is said to be an $(\calA, G)$-module if $g(am)=(ga)(gm)$ for any $g \in G$, $a \in \calA$ and $m \in M$.
Moreover, we will say that an $(\calA, G)$-module $M$ is \define{finitely generated} if
$M$ is finitely generated as an $\calA$-module.
\end{definition}

Let $X$ be a quasi-projective variety over $\mathbb{C}$.
We denote by $\rring{X}$ the ring of regular functions on $X$.
Suppose $G$ acts on $X$ rationally.
The action of $G$ on $X$ induces a rational representation of $G$ on $\rring{X}$ as follows:
\begin{align*}
g \cdot f(x) = f(g^{-1}x) \text{ for }g \in G, f \in \rring{X}.
\end{align*}
We write $\dweight(X) = \dweight(\rring{X})$ for short.

\subsection{Some finiteness results}
We prepare some finiteness results.
Let $G$ be a connected reductive algebraic group over $\CC$,
and $B=TN$ be a Borel subgroup of $G$.
The following result is due to D{\v{z}}. Had{\v{z}}iev and F. D. Grosshans \cite{Gr83}.

\begin{proposition}\label{proposition:finiteness of N-invariant algebra}
Let $\calA$ be a $G$-algebra.
Then, $(\calA \otimes \rring{G/N})^{G}$ is isomorphic to $\calA^{N}$ as a $\CC$-algebra.
Moreover if $\calA$ is finitely generated, $\calA^{N}$ is finitely generated.
\end{proposition}

\begin{remark}
For the following lemmas, we only define the isomorphism.
Since $\calA$ is a rational representation of $G$,
$\varphi_a(g):=ga$ is well-defined as a element of $\calA \otimes \rring{G}$ for any $a \in \calA$.
The image of $\varphi$ is contained in $(\calA \otimes \rring{G})^{G}$.
Since $\varphi_{ha}(g)=g h a=\varphi_a(gh)$, taking $N$-invariant part,
we have $\varphi: \calA^{N} \rightarrow (\calA \otimes \rring{G/N})^{G}$.
A map $f \mapsto f(eN)$ gives the inverse of $\varphi$.
By the definition of a $G$-algebra, it is obvious that $\varphi$ is homomorphism of a $\CC$-algebra.
\end{remark}


\begin{lemma}\label{lemma:finiteness of G-invariant module}
Let $\calA$ be a Noetherian $G$-algebra, and $M$ be a finitely generated $(\calA, G)$-module.
Then, $M^{G}$ is also a finitely generated $\calA^{G}$-module.
\end{lemma}

\begin{proof}
Since $\calA$ is a Noetherian algebra and $M$ is finitely generated, $M$ is a Noetherian $\calA$-module.
Then, $\calA M^{G}$ is finitely generated.

Let $\{m_1, m_2, \cdots, m_r\}$ be a finite generating set of $\calA M^{G}$.
We can assume $\{m_1, m_2, \ldots, m_r\} \subset M^{G}$.
Let us show that $\{m_1, m_2, \ldots, m_r\}$ is also a generating set of $M^{G}$ as an $\calA^{G}$-module.
For any $m \in M^{G}$, there exist $f_1, f_2, \ldots, f_r \in \calA$ such that
$m = f_1 m_1 + f_2 m_2 + \cdots f_r m_r$.
Taking $G$-invariant part, we have $m = f_1^{G} m_1 + f_2^{G} m_2 + \cdots f_r^{G} m_r$,
where $f_i^{G}$ is the projection to $G$-invariant part of $f_i$.
This shows $\{m_1, m_2, \ldots, m_r\}$ generates $M^{G}$ as an $\calA^{G}$-module.
\end{proof}

The following lemma is a key result for the proof of Theorem \ref{theorem:general stability}.
If $\calA$ is finitely generated, this result (for arbitrary characteristics) was appeared in \cite{Gr92}.

\begin{lemma}\label{lemma:finiteness of N-invariant module}
Let $\calA$ be a Noetherian $G$-algebra, and $M$ be an $(\calA, G)$-module.
Then, $M^{N}$ is isomorphic to $(M \otimes \rring{G/N})^{G}$ as an $\calA^{N}$-module.
Here, we consider $(M \otimes \rring{G/N})^{G}$ as an $\calA^{N}$-module
via isomorphism $\calA^{N} \simeq (\calA \otimes \rring{G/N})^{G}$ in Proposition \ref{proposition:finiteness of N-invariant algebra}.
Moreover, if $M$ is a finitely generated $\calA$-module, then $M^{N}$ is also a finitely generated $\calA^{N}$-module.
\end{lemma}

\begin{proof}
Since $\rring{G/N}$ is finitely generated, $\calA \otimes \rring{G/N}$ is a Noetherian algebra from Hilbert's basis theorem.
Then, the second argument is followed from the first argument and Lemma \ref{lemma:finiteness of G-invariant module}.

For the first argument, it is suffices to define the isomorphism between $M^{N}$ and $(M \otimes \rring{G/N})^{G}$.
Since $M$ is a rational representation of $G$, $\varphi_m(g):=gm$ is well-defined as a element of $M \otimes \rring{G}$ for any $m \in M$.
The image of $\varphi$ is contained in $(M \otimes \rring{G})^{G}$.
Since $\varphi_{hm}(g)=g h m=\varphi_m(gh)$, taking $N$-invariant part, we have $\varphi: M^{N} \rightarrow (M \otimes \rring{G/N})^{G}$.
A map $f \mapsto f(eN)$ gives the inverse of $\varphi$.
Then, $M^{N}$ and $(M \times \rring{G/N})^{G}$ are isomorphic as vector spaces.
By the definition of the isomorphism between $\calA^{N}$ and $(\calA \otimes \rring{G/N})^{G}$,
the vector space isomorphism $M^{N} \simeq (M \times \rring{G/N})^{G}$ is also an $\calA^{N}$-module isomorphism.
Then, this completes the proof.
\end{proof}

\section{Preliminaries: Highest weight modules}

Let $G$ be a real reductive Lie group, and $K$ be a maximal compact subgroup of $G$.
Let $\frakg = \frakk \oplus \frakp$ be the Cartan decomposition of $\frakg$ determined by $K$.
Since $K$ is a compact Lie group, there exists a complexification $K_{\CC}$ of $K$,
and $K_{\CC}$ has a complex reductive algebraic group structure.
Moreover, any locally finite representations of $K$ can be extended to rational representations of $K_{\CC}$.

\subsection{Associated variety and isotropy representation}\label{subsection:associated_variety}

Let $V$ be a finitely generated $(\frakg, K)$-module.
Since $V$ is finitely generated, we can take a $K$-invariant finite dimensional subspace $W \subset V$ which generates $V$.
We put $V_i := \calU_i(\frakg_{\CC})W$ and $V_{-1}:=0$,
where $\{\calU_i(\frakg_{\CC})\}$ is the canonical filtration of the universal enveloping algebra $\calU(\frakg_{\CC})$.
Taking the associated graded module, we have an $(S(\frakg_{\CC}), K_{\CC})$-module
\begin{align*}
\gr(V):=\bigoplus^{\infty}_{i=0} V_{i}/V_{i-1}.
\end{align*}
An affine variety determined by $\Ann_{S(\frakg_{\CC})}(\gr(V))$ is called the \define{associated variety of $V$},
and denoted by $\AV(V) \subset \frakg_{\CC}^{*}$.
It is well-known that $\AV(V)$ is independent of the choice of $W$.
Since the filtration is $K_{\CC}$-stable, $\AV(V)$ is $K_{\CC}$-stable variety contained in $(\frakg_{\CC}/\frakk_{\CC})^{*}$.
We identify $\frakg_{\CC}^{*}$ and $\frakg_{\CC}$ by some invariant bilinear form of $\frakg_{\CC}$.
By this identification, $(\frakg_{\CC}/\frakk_{\CC})^{*}$ corresponds to $\frakp_{\CC}$,
and $\AV(V)$ becomes a subvariety in the nilpotent cone in $\frakp_{\CC}$.
It is known that the the number of $K_{\CC}$-orbits in the nilpotent cone in $\frakp_{\CC}$ is finite.
Then, there exists an open $K_{\CC}$-orbit in $\AV(V)$.

After here, we assume that $\AV(V)$ is irreducible for convenience.
Let us define the isotropy representation of $V$ introduced by D. Vogan \cite{Vo89} (see also \cite{Ya05}).
Let $I$ be the defining ideal of $\AV(V)$.
By the Hilbert Nullstellensatz, $I^n$ is contained in $\Ann_{S(\frakg_{\CC})}(\gr(V))$ for some positive integer $n$.
Since $\AV(V)$ is irreducible, $\AV(V)$ has a unique open dense $K_{\CC}$-orbit $\calO$.
Fix a point $x_0 \in \calO$.
We denote by $\frakm(x_0) \subset S(\frakg)$ the maximal ideal corresponding to $x_0$.
We set
\begin{align*}
\calW := \calW(x_0) = \bigoplus_{i=0}^{n-1}I^{i}V/\frakm(x_0)I^{i}V.
\end{align*}
$\calW$ becomes a finite dimensional rational representation of $(K_{\CC})_{x_0}$,
where $(K_{\CC})_{x_0}$ is the isotropy subgroup of $K_{\CC}$ at $x_0$.
The representation $\calW$ is called the \define{isotropy representation} of $V$.
Note that the isotropy representation is dependent on the filtration of $V$ and the point $x_0$.

\subsection{Highest weight modules}
Suppose $G$ is a connected non-compact simple Lie group with finite center.
Though the assumption `finite center' is not essential,
we assume this for convenience.
We assume that $(\frakg, \frakk)$ is a Hermitian symmetric pair (i.e. the center $Z(\frakk)$ of $\frakk$ is one-dimensional).
We fix a \define{characteristic element} $Z \in Z(\frakk_{\CC})$ such that the eigenvalues of $\ad(Z)$ are $0, \pm 1$,
and we write its eigenspace decomposition as
\begin{align*}
\frakg_{\CC} = \frakp_{+} \oplus \frakk_{\CC} \oplus \frakp_{-},
\end{align*}
with the eigenvalues $1, 0, -1$, respectively.

For an irreducible $(\frakg, K)$-module $V$,
we will say $V$ is a \define{highest weight module} of $G$ if $\frakp_{+}$-nullvectors $V^{\frakp_{+}} \neq 0$.
Moreover, if $V$ is infinitesimally unitary, we will say $V$ is a \define{unitary highest weight module}.

If $\calH$ is a highest weight module, $\calH^{\frakp_{+}}$ is an irreducible representation of $K$, and
$\calH^{\frakp_{+}}$ generates $\calH$ as a representation of $\calU(\frakg_{\CC})$.
For any highest weight module $\calH$, we always take $\calH^{\frakp_{+}}$ as $W$ in Section \ref{subsection:associated_variety}.
Since the filtration determined by $W = \calH^{\frakp_{+}}$ is stable under $\frakp_{+}$-action,
its associated graded module $\gr(\calH)$ becomes naturally an $(S(\frakp_{-}), K_{\CC})$-module.
Then, the associated variety of $\calH$ is contained in $\frakp_{+}$.
Since the graded module $\gr(\calH)$ is isomorphic to $\calH$ with a grading $\calH^{i}:=S^{i}(\frakp_{-})\calH^{\frakp_{+}}$,
we do not care about the filtration step when we consider highest weight modules.

About the annihilators of unitary highest weight modules, A. Joseph showed the following result in \cite{Jo92}:
\begin{proposition}\label{proposition:joseph}
Let $\calH$ be a unitary highest weight module.
Then, the annihilator $\Ann_{S(\frakp_{-})}(\calH)$ is a prime ideal in $S(\frakp_{-})$,
and $\Ann_{S(\frakp_{-})}(v) = \Ann_{S(\frakp_{-})}(\calH)$ for any $v \in \calH$.
\end{proposition}

By this proposition, the isotropy representation of a highest weight module at $x_0 \in \AV(\calH)$ is simply written as $\calW = \calH/\frakm(x_0)\calH$.

Since $\calH^{\frakp_{+}}$ generates $\calH$, we have a canonical surjective homomorphism as $(\frakg, K)$-module:
\begin{align}
\calU(\frakg_{\CC}) \otimes_{\calU(\frakk_{\CC} \oplus \frakp_{+})} \calH^{\frakp_{+}} \rightarrow \calH. \label{equation:canonical_surjection}
\end{align}
For a finite dimensional representation $V$ of $K$,
we set
\begin{align*}
N^\frakg(V):=\calU(\frakg_{\CC}) \otimes_{\calU(\frakk_{\CC} \oplus \frakp_{+})} V.
\end{align*}

\begin{definition}\label{definition:singular_highest_weight_module}
Let $\calH$ be a unitary highest weight module.
We will say $\calH$ is a \define{holomorphic discrete series representation} if
the completion of $\calH$ with respect to its Hermitian inner product is a discrete series of $G$.
\end{definition}
It is known that if $\calH$ is a holomorphic discrete series representation,
the homomorphism (\ref{equation:canonical_surjection}) is a $(\frakg, K)$-module isomorphism.
Then, for a holomorphic discrete series representation $\calH$, the associated variety $\AV(\calH)$ is equal to $\frakp_{+}$.

\subsection{Strongly orthogonal roots}\label{section:strongly_orthogonal}
We will describe some structures of highest weight modules.
We take a Cartan subalgebra $\frakt \subset \frakk$.
Since $\frakg$ is Hermitian type, $\frakt$ is also a Cartan subalgebra of $\frakg$.
Let $\Delta:=\Delta(\frakg_{\CC}, \frakt_{\CC})$ be the root system determined by $\frakt_{\CC}$,
and fix a positive system $\Delta^{+}$ such that $\Delta^{+} \supset \Delta(\frakp_{+},\frakt_{\CC})$.
We write $\Delta^{+}_c := \Delta(\frakk_{\CC}, \frakt_{\CC}) \cap \Delta^{+}$ and $\Delta^{+}_n := \Delta(\frakp_{+},\frakt_{\CC})$.
For each $\lambda \in \frakt_{\CC}^{*}$, we set
\begin{align*}
\frakg_{\lambda}:=\frakg_{\CC}(\frakt_{\CC};\lambda):=\{X \in \frakg_{\CC}: [H,X]=\lambda(H)X \text{ for any }H \in \frakt_{\CC}\}.
\end{align*}

Two roots $\alpha, \beta$ are said to be \define{strongly orthogonal} if neither of $\alpha+\beta$ nor $\alpha-\beta$ is a root.
We take a maximal set of strongly orthogonal roots $\{\sor_1, \sor_2, \ldots, \sor_r\} \subset \Delta(\frakp_{+}, \frakt_{\CC})$ as follows:
\begin{enumerate}[i)]
	\item $\sor_1$ is the lowest root in $\Delta(\frakp_{+}, \frakt_{\CC})$,
	\item for $i>1$, $\sor_i$ is the lowest root in the roots that are strongly orthogonal to $\sor_1, \sor_2, \ldots, \sor_{i - 1}$.
\end{enumerate}
Fix root vectors $\{X_{\sor_i}\}_{i=1}^r$ for the roots $\{\sor_i\}_{i=1}^{r}$.
We set 
\begin{align*}
\fraka &:= \bigoplus_{i=1}^{r} \RR (X_{\sor_i}+\overline{X_{\sor_i}}), \\
\frakt_0 &:= \bigoplus_{i=1}^{r} \CC [X_{\sor_i},\overline{X_{\sor_i}}],
\end{align*}
where $\overline{\ \cdot\ }$ is a complex conjugate of $\frakg_{\CC}$ with respect to $\frakg$.
It is known that $\fraka$ becomes a maximal abelian subspace of $\frakp$.
Then, we have $r = \RR \text{-} \rank(\frakg)$.

We introduce some facts to describe the restricted roots of $G$.
For $i, j\ (1 \leq i < j \leq r)$, we put
\begin{align*}
C_{ij} &:= \left\{ \gamma \in \Delta^{+}_c: \gamma|_{\frakt_0} = \left.\left(\frac{\gamma_j - \gamma_i}{2}\right) \right|_{\frakt_0} \right\}, \\
C_{i} &:= \left\{ \gamma \in \Delta^{+}_c: \gamma|_{\frakt_0} = -\left.\left(\frac{\gamma_i}{2}\right) \right|_{\frakt_0} \right\}, \\
C_{0} &:= \{ \gamma \in \Delta^{+}_c: \gamma|_{\frakt_0} = 0 \}. \\
P_{ij} &:= \left\{ \gamma \in \Delta^{+}_n: \gamma|_{\frakt_0} = \left.\left(\frac{\gamma_j + \gamma_i}{2}\right)\right|_{\frakt_0} \right\}, \\
P_{i} &:= \left\{ \gamma \in \Delta^{+}_n: \gamma|_{\frakt_0} = \left.\left(\frac{\gamma_i}{2}\right) \right|_{\frakt_0} \right\}, \\
P_{0} &:= \{ \gamma_1, \gamma_2, \ldots, \gamma_r \}.
\end{align*}

The following fact is due to Moore. (see e.g. \cite[Proposition 4.8 in Chapter 5]{He94}).

\begin{proposition}\label{proposition:restricted_root}
In the above notation, $\Delta^{+}_c$ and $\Delta^{+}_n$ can be decomposed as follows:
\begin{align*}
\Delta^{+}_c &= \left(\bigcup_{1 \leq i<j \leq r}C_{ij}\right) \cup \left(\bigcup_{1 \leq i \leq r}C_{i}\right) \cup C_0, \\
\Delta^{+}_n &= \left(\bigcup_{1 \leq i<j \leq r}P_{ij}\right) \cup \left(\bigcup_{1 \leq i \leq r}P_{i}\right) \cup P_0.
\end{align*}
Moreover, the map $\gamma \mapsto \gamma + \gamma_i$ gives bijections from $C_{ij}$ to $P_{ij}$, from $-C_{ji}$ to $P_{ji}$, and from $C_{i}$ to $P_{i}$.
\end{proposition}

It is known that $K_{\CC}$-orbits in $\frakp_{+}$ can be described by strongly orthogonal roots.
Put $X_i = X_{\sor_1} + X_{\sor_2} + \cdots + X_{\sor_i}$.
We set $\calO_i := \Ad(K_{\CC})X_i$, and $\calO_0 = \{0\}$.

\begin{proposition}\label{proposition:kc-orbits}
$\frakp_{+}$ is decomposed into $K_{\CC}$-orbits as follows:
\begin{align*}
\frakp_{+} = \coprod_{i=0}^{n}\calO_i.
\end{align*}
Moreover, for any $1\leq m \leq r$, the Zariski closure of $\calO_m$ is decomposed into $K_{\CC}$-orbits as follows:
\begin{align*}
\overline{\calO_m} = \coprod_{i=0}^{m}\calO_i.
\end{align*}
\end{proposition}

By this proposition, for any highest weight module $\calH$ there exists an $m \in \{1, 2, \cdots, r\}$ such that $\AV(\calH)=\overline{\calO_m}$.
The irreducible decomposition of $\rring{\overline{\calO_m}}$ as a representation of $K_{\CC}$
is obtained by B. Kostant, L. K. Hua \cite{Hu63} and W. Schmid \cite{Sh69}.

\begin{proposition}\label{proposition:kostant_hua_schmid}
The ring of regular functions on $\overline{\calO_m}$ is decomposed as a $K_{\CC}$-representation as follows:
\begin{align*}
\rring{\overline{\calO_m}} \simeq \bigoplus_{\substack{c_1 \geq c_2 \geq \cdots \geq c_m \geq 0 \\ c_1, c_2, \ldots, c_m \in \ZZ}} V_{-\sum_{i=1}^{m}c_i \sor_i, K_{\CC}}.
\end{align*}
Especially, $\overline{\calO_m}$ is a spherical affine $K_{\CC}$-variety.
\end{proposition}

\subsection{Holomorphic symmetric pairs}\label{section:holomorphic_symmertic_pair}
Suppose $\theta$ is a Cartan involution of $G$ such that its fixed point subgroup $G^{\theta} = K$, and
$\tau$ is an involutive automorphism of $G$ commuting with $\theta$.
Since $\tau(\frakk) = \frakk$ and $\tau$ is an automorphism, the following two cases are possible:
\begin{align}
\tau(Z) &= Z \label{equation:holomorphic_symmetric_pair} \\
\tau(Z) &= -Z \label{equation:anti_holomorphic_symmetric_pair}
\end{align}
We will say $(\frakg, \frakg^{\tau})$ is a \define{holomorphic symmetric pair}
if the equation (\ref{equation:holomorphic_symmetric_pair}) holds,
otherwise we will say $(\frakg, \frakg^{\tau})$ is an \define{anti-holomorphic symmetric pair}.

If $(\frakg, \frakg^{\tau})$ is a holomorphic symmetric pair,
the decomposition $\frakg_{\CC}=\frakp_{-} \oplus \frakk_{\CC} \oplus \frakp_{+}$ induces a decomposition of $\frakg_{\CC}^{\tau}$:
\begin{align*}
\frakg_{\CC}^{\tau}=\frakp_{-}^{\tau} \oplus \frakk_{\CC}^{\tau} \oplus \frakp_{+}^{\tau},
\end{align*}
since $Z \in \frakk_{\CC}^{\tau}$.
Suppose $\frakg^{\tau}=\bigoplus_{i = 1}^{n} \frakh_i$ is the direct sum decomposition into simple or abelian ideals.
Then, $\frakh_i$ is contained in $\frakk$ if $\frakh_i$ is a compact or abelian Lie algebra,
and $\frakh_i$ is a Hermitian type Lie algebra if $\frakh_i$ is a non-compact Lie algebra.
Moreover, if $\frakh_i$ is a Hermitian type Lie algebra, $\frakh_i$ has a decomposition:
\begin{align}
(\frakh_i)_{\CC}=(\frakp_{-}\cap (\frakh_i)_{\CC}) \oplus (\frakk_{\CC}\cap (\frakh_i)_{\CC}) \oplus (\frakp_{+} \cap (\frakh_i)_{\CC}), \label{equation:direct_sum_of_h}
\end{align}
and each summand is nonzero.
The following two facts about branching laws of the restriction with respect to holomorphic symmetric pairs are known (see e.g. \cite{Ko94, Ko98}).

\begin{proposition}\label{proposition:discretely_decomposable}
Let $\calH$ be a unitary highest weight module of $G$,
and $(\frakg, \frakg^{\tau})$ be a holomorphic symmetric pair.
Then, $\calH$ is $\frakg^{\tau}$-admissible, especially $\calH$ is discretely decomposable as $(\frakg^{\tau}, \frakk^{\tau})$-module.
Moreover, any irreducible components of $\calH|_{\frakg^{\tau}}$ are outer tensor products of highest weight modules and finite dimensional representations.
\end{proposition}



\begin{proposition}\label{proposition:reduction_to_compact}
Let $\calH$ be a holomorphic discrete series representation of $G$,
and $(\frakg, \frakg^{\tau})$ be a holomorphic symmetric pair.
Suppose $S(\frakp_{-}^{-\tau}) \otimes \calH^{\frakp_{+}}$ is decomposed as a $(\frakk^{\tau})$-representation as follows:
\begin{align*}
S(\frakp_{-}^{-\tau}) \otimes \calH^{\frakp_{+}} \simeq \bigoplus_{\pi \in \widehat{\frakk^{\tau}}} m(\pi)\pi.
\end{align*}
Then, $\calH|_{\frakg^{\tau}}$ is decomposed as
\begin{align*}
\calH|_{\frakg^{\tau}} \simeq \bigoplus_{\pi \in \widehat{\frakk^{\tau}}} m(\pi)(N^{\frakg^{\tau}}(\pi)).
\end{align*}
Here, $\widehat{\frakk^{\tau}}$ denotes the set of equivalent classes of finite dimensional representations of $\frakk^{\tau}$.
Each summand is also a holomorphic discrete series representation of $(G^{\tau})_0$.
\end{proposition}

\begin{table}[h]
\centering
\caption{holomorphic symmetric pairs}
\begin{tabular}{|c|c|}
\hline
$\frakg$& $\frakg^{\tau}$ \\ \hline \hline
$\su(p,q)$& $\fraks(\uor(i,j)+\uor(p-i,q-j))$ \\ \hline
$\su(n,n)$& $\so^{*}(2n)$ \\ \hline
$\su(n,n)$& $\spor(n,\RR)$ \\ \hline
$\so^{*}(2n)$& $\uor(i,n-i)$ \\ \hline
$\so^{*}(2n)$& $\so^{*}(2i)+\so^{*}(2(n-i))$ \\ \hline
$\so(2,n)$& $\so(2,n-i)+\so(i)$ \\ \hline
$\so(2,2n)$& $\uor(1,n)$ \\ \hline
$\spor(n,\RR)$& $\uor(i,n-i)$ \\ \hline
$\spor(n,\RR)$& $\spor(i,\RR)+\spor(n-i,\RR)$ \\ \hline
$\frake_{6(-14)}$&$\so(10)+\so(2)$ \\ \hline
$\frake_{6(-14)}$&$\so^{*}(10)+\so(2)$ \\ \hline
$\frake_{6(-14)}$&$\so(8,2)+\so(2)$ \\ \hline
$\frake_{6(-14)}$&$\su(5,1)+\slor(2,\RR)$ \\ \hline
$\frake_{6(-14)}$&$\su(4,2)+\su(2)$ \\ \hline
$\frake_{7(-25)}$&$\frake_{6(-78)}+\so(2)$ \\ \hline
$\frake_{7(-25)}$&$\frake_{6(-14)}+\so(2)$ \\ \hline
$\frake_{7(-25)}$&$\so(10,2)+\slor(2,\RR)$ \\ \hline
$\frake_{7(-25)}$&$\so^{*}(12)+\su(2)$ \\ \hline
$\frake_{7(-25)}$&$\su(6,2)$ \\ \hline
\end{tabular}
\end{table}

\section{Stability theorem}

In this section, we will show a general stability theorem.

\subsection{Stability theorem for general settings}

Let $X$ be an irreducible quasi-projective variety over $\CC$.
We assume the following two conditions:
\begin{texteqn}
$X$ is a spherical $G$-variety (i.e. a Borel subgroup of $G$ has an open dense orbit in $X$), and \label{condition:mf}
\end{texteqn}
\begin{texteqn}
the quotient field of $\rring{X}$ is naturally isomorphic to the rational function field of $X$. \label{condition:rf}
\end{texteqn}

The first condition implies that $\rring{X}$ is multiplicity-free as a representation of $G$.
Note that the second condition is always true for any irreducible quasi-affine variety $X$.

\begin{theorem}\label{theorem:general stability}
Let $M$ be a finitely generated $(\rring{X}, G)$-module
with no zero divisors:
\begin{equation}
\bigcup_{m \in M \setminus \{0\}} \Ann_{\rring{X}}(m) = 0.
\label{equation:stability assumption}
\end{equation}
Then, there exists a weight $\lambda_0 \in \dweight(X)$ such that
\begin{equation*}
m_M(\lambda+\lambda_0) = m_M(\lambda+\mu+\lambda_0)
\end{equation*}
for any $\lambda \in \dweight(M)$ and $\mu \in \dweight(X)$.
\end{theorem}

This theorem says that the multiplicity function $m_M$ is periodic for sufficiently large parameter $\lambda$.
This property of the multiplicity function is called stability.

\begin{proof}
For the proof of the stability, we will show the uniformly boundedness of $m_M(\lambda)$ for $\lambda \in \dweight$.
Since $M$ is a finitely generated $\rring{X}$-module, there exists a $G$-invariant finite dimensional subspace $F \subset M$
that generates $M$.
Then, the multiplication map $\rring{X} \otimes F \rightarrow M (f\otimes m \mapsto fm)$ is a surjective $G$-intertwining operator.
For any $\lambda \in \dweight(M)$, we have
\begin{align*}
m_M(\lambda) &\leq m_{\rring{X} \otimes F}(\lambda) \\
&= \dim \Hom_G(V_\lambda, \rring{X} \otimes F) \\
&= \dim \Hom_G(V_\lambda \otimes F^{*}, \rring{X}).
\end{align*}
From \cite[Proposition 5.4.1]{Ko08}, the number of the irreducible constituents of $V_\lambda \otimes F^{*}$ is bounded by $\dim(F)$.
Then, since $\rring{X}$ is multiplicity-free, $m_M(\lambda)$ is uniformly bounded by $\dim(F)$.
This result will be also proved in the proof of Theorem \ref{theorem:stable_multiplicity_M}.

Let us show the stability.
Since $\rring{X}$ has no zero divisors in $M$,
the multiplication operator ($m \mapsto fm$) is injective for any $f \in \rring{X}$.
Especially, for $\mu \in \dweight(X)$ and $f \in \rring{X}^{N}(\mu)$, $f$ induces an injective linear map
\begin{align*}
f \cdot : M^{N}(\lambda) \hookrightarrow M^{N}(\lambda+\mu),
\end{align*}
where $V(\lambda)$ denotes the weight space of weight $\lambda$ in $T$-representation $V$.
Then, we have
\begin{align}
m_M(\lambda) \leq m_M(\lambda+\mu) \label{equation:proof_stability_monotone}
\end{align}
for any $\mu \in \dweight(X)$.

From Lemma \ref{lemma:finiteness of N-invariant module}, $M^{N}$ is a finitely generated $\rring{X}^{N}$-module.
Then, we can take a finite subset $\{\lambda_1, \lambda_2, \ldots, \lambda_r\} \subset \dweight(M)$ such that
\begin{align}
\dweight(M) = \bigcup_{1 \leq i \leq r}(\dweight(X) + \lambda_i). \label{equation:proof_stability_covering}
\end{align}
By the uniformly boundedness of $m_M$, for each $\lambda_i$ we can find a $\lambda_{0,i} \in \dweight(X)$ such that
\begin{align}
m_M(\lambda_i+\lambda_{0,i})=\max\{m_M(\lambda_i+\mu): \mu \in \dweight(X)\}. \label{equation:proof_stability_maximality}
\end{align}

We put $\lambda_0 := \lambda_{0,1} + \lambda_{0,2} + \ldots + \lambda_{0,r}$.
Let us show that $\lambda_0$ satisfies the required condition.
Take $\lambda \in \dweight(M)$ and $\mu \in \dweight(X)$.
By (\ref{equation:proof_stability_covering}), there exists an $i \in \{1, 2, \ldots, r\}$ such that
$\lambda \in \lambda_i + \dweight(X)$.
From (\ref{equation:proof_stability_monotone}) and (\ref{equation:proof_stability_maximality}),
we have
\begin{align*}
m_M(\lambda+\lambda_0)&=\max\{m_M(\lambda_i+\mu): \mu \in \dweight(X)\} \\
&=m_M(\lambda+\mu+\lambda_0).
\end{align*}
This shows the theorem.
\end{proof}

\subsection{Description of multiplicities for large parameters}\label{section:stable_multiplicity}

We describe the multiplicities for sufficiently large parameters by the isotropic representation.
Let $G$ and $X$ be as in the previous section.
Take a Borel subgroup $B$ of $G$, and write $B=TN$ as the product of the maximal torus and the unipotent radical.
From the assumption (\ref{condition:mf}), there exists a point $x_0 \in X$ such that $B$-orbit $Bx_0$ is open dense in $X$.

Put $P = \{g \in G: gBx_0 \subset Bx_0\}$. Then, $P$ is a parabolic subgroup of $G$ contains $B$.
The following proposition is due to M. Brion, D. Luna and T. Vust \cite{BLV86}.

\begin{proposition}\label{proposition:BLV}
In the above settings,
\begin{enumerate}[i)]
	\item $P_{x_0}$ is a reductive subgroup of $G$,
	\item $P_{x_0}$ contains the derived group of some Levi subgroup of $P$.
\end{enumerate}
\end{proposition}

The following proposition says that $B_{x_0}$ is a 'Borel subgroup' of $P_{x_0}$ in some sense.

\begin{proposition}\label{proposition:L-subgroup}
$P_{x_0}$ satisfies the following four conditions:
\begin{enumerate}[L-1)]
	\item $P_{x_0} \subset G_{x_0}$,
	\item $P_{x_0} \supset B_{x_0}$,
	\item $B_{x_0}$ meets every connected components of $P_{x_0}$,
	\item the identity component of $B_{x_0}$ is a Borel subgroup of the identity component of $P_{x_0}$.
\end{enumerate}
Conversely, if a reductive subgroup $L$ of $G$ satisfies the above four conditions instead of $P_{x_0}$, then we have $L = P_{x_0}$.
\end{proposition}

\begin{remark}\label{remark:irreducible_of_L}
If $L$ satisfies the above four condition, its irreducible representations are parametrized by a subset of characters of $B_{x_0}$.
This is because $V^{N_{x_0}}$ is one-dimensional for any irreducible representation $V$ of $L$.
In fact, since we have the natural injection $B_{x_0}/N_{x_0} \hookrightarrow B/N \simeq T$, we can take a weight vector $v \in V^{N_{x_0}}$ with respect to $B_{x_0}/N_{x_0}$.
$v$ generates an irreducible representation $V_0$ of $L_0$, where $L_0$ is the identity component of $L$.
Since $B_{x_0}$ normalizes $L_0$, $V_0$ is $L$-stable.
This shows $V_0 = V$.
Then, $V^{N_{x_0}}$ is one-dimensional.
\end{remark}

\begin{proof}
For the first argument, put $L := P_{x_0}$.
By definition, L-1) and L-2) are clear.
From Proposition \ref{proposition:BLV}, we can take a Levi subgroup $Q$ of $P$ such that $[Q, Q]$ is contained in $L$.
We have the following commutative diagram.
\begin{align*}
\xymatrix{
L/(L \cap B) \ar@{^{(}->}[r] & P/B \\
[Q,Q]/([Q,Q]\cap B) \ar@{^{(}->}[u] \ar[ru]^{\simeq} & \\
}
\end{align*}
Then, $L/(L \cap B)$ is isomorphic to $P/B$.
Since $P/B \simeq Q/Q\cap B$ is a connected projective variety, $B_{x_0} = L \cap B$ meets every connected components of $L$,
and the identity component of $B_{x_0}$ is a Borel subgroup of the identity component of $L$.
This implies that $L$ satisfies L-3) and L-4).

For the second argument, suppose $L$ is a reductive subgroup of $G$ that satisfies the conditions.
From Remark \ref{remark:irreducible_of_L}, we have
\begin{align}
\rring{G}^{L}&= \rring{G}^{B_{x_0}} \nonumber \\
&= \rring{G}^{P_{x_0}}. \label{equation:L=P_x}
\end{align}
For a reductive subgroup $H$ of $G$, $H$ can be reconstructed from $\rring{G}^{H}$ by
the following equation:
\begin{align*}
H=\bigcap_{f \in \rring{G}^{H}} f^{-1}(f(e)).
\end{align*}
Here, $e$ is the identity of $G$.
From this fact and (\ref{equation:L=P_x}), we have $L=P_{x_0}$.
This completes the proof.
\end{proof}

We set $L=P_{x_0}$.
We denote by $\ev_{x_0}$ the natural quotient map $M \rightarrow M_{x_0}(:=M/\frakm({x_0})M)$.
From the inclusion $L \subset G_{x_0}$, $\ev_{x_0}$ is $L$-intertwining operator from $M$ to $M_{x_0}$.
We describe the stable multiplicities by the representation of $L$ on $M/\frakm({x_0})M$.

\begin{theorem}\label{theorem:stable_multiplicity_M}
Let $M$ be a finitely generated $(\rring{X}, G)$-module with no zero divisors (see (\ref{equation:stability assumption})).
We take a weight $\lambda_0 \in \dweight(X)$ described in Theorem \ref{theorem:general stability}.
Then, for any $\lambda \in \dweight(M)$
\begin{align*}
m_M^{G}(\lambda+\lambda_0)=m_{M_{x_0}}^{L}(\lambda|_{B_{x_0}}).
\end{align*}
Here, we identify characters of $T$ and characters of $B$ by letting their values be $1$ on $N$.
\end{theorem}

\begin{remark}
For any $\mu \in \dweight(X)$, $\mu|_{B_{x_0}} = 0$.
Then, $\lambda|_{B_{x_0}}$ can also be written as $(\lambda+\lambda_0)|_{B_{x_0}}$.
\end{remark}

\begin{lemma}\label{lemma:stable_multiplicity_surjective}
Under the assumption (\ref{condition:mf}) and (\ref{condition:rf}), we have the following equation:
\begin{align*}
\rring{B{x_0}} = \rring{X}\left[\frac{1}{f_\mu}: \mu \in \dweight(X), f_{\mu} \in \rring{X}^{N}(\mu) \backslash \{0\} \right].
\end{align*}
\end{lemma}

\begin{proof}
This lemma is essentially same as \cite[Lemma 2.2]{Sa93}.
It is clear that the left hand side contains the right hand side.
We will show the converse inclusion.

We take a function $f \in \rring{B{x_0}}$.
Define an ideal as
\begin{align*}
I:=\{g \in \rring{X}: g\cdot bf \in \rring{X} \text{ for any }b \in B\}.
\end{align*}
Since $B$ acts rationally on $\rring{B{x_0}}$, $\spn{bf: b \in B}$ is finite dimensional.
By the assumption (\ref{condition:rf}), $I$ is a $B$-invariant nonzero ideal of $\rring{X}$.
Since $B$ acts rationally on $I$, there exists a nonzero $B$-eigenvector $g \in I$.
Then, we have $f \in \rring{X}[1/g]$.
This shows the converse inclusion.
\end{proof}

\begin{lemma}\label{lemma:stable_multiplicity_injective}
Let $M$ be a $(\rring{X},G)$-module.
Suppose $\rring{X}$ has no zero divisors in $M$.
Then, we have 
\begin{align*}
\bigcap_{y \in B{x_0}}(\frakm(y)M) = 0.
\end{align*}
\end{lemma}

\begin{proof}
If $M$ is finitely generated, this lemma is in \cite[Corollary 2.1]{Ya05}.
Put $N = \bigcap_{y \in B{x_0}}(\frakm(y)M)$.
We assume $N \neq 0$.
Since $N$ is $B$-invariant subspace, there exists a nonzero $B$-eigenvector $m \in N$.
By definition, $m$ can be written as
\begin{align}
m = f_1 m_1 + f_2 m_2 + \cdots + f_r m_r \ \ \ (f_i \in \frakm({x_0}), m_i \in M). \label{equation:m_in_injective}
\end{align}
Let $M'$ be a $(\rring{X}, G)$-submodule of $M$ generated by $m_1, m_2, \ldots, m_r$.
Since $M'$ is a finitely generated $(\rring{X},G)$-module, we have
\begin{align*}
\bigcap_{y \in B{x_0}}(\frakm(y)M')=0.
\end{align*}
From (\ref{equation:m_in_injective}), $m$ is an element of $\frakm({x_0})M'$.
Since $m$ is a $B$-eigenvector, we have $m \in \bigcap_{y \in B{x_0}}(\frakm(y)M')$
and then $m = 0$.
However, this contradicts the assumption that $m$ is nonzero.
Thus, Lemma \ref{lemma:stable_multiplicity_injective} is proved.
\end{proof}

\begin{proof}[Proof of Theorem \ref{theorem:stable_multiplicity_M}]
Take $\lambda \in \dweight(M)$.
From Remark \ref{remark:irreducible_of_L}, we have
\begin{align*}
m_{M/\frakm({x_0})M}^{L}(\lambda|_{B_{x_0}}) = \dim((M/\frakm({x_0})M)^{N_{x_0}}(\lambda|_{B_{x_0}}))
\end{align*}
Since $\ev_{x_0}$ is $G_x$-intertwining operator, the image of $M^{N}(\lambda+\lambda_0)$ by $\ev_{x_0}$ is contained in $(M/\frakm({x_0})M)^{N_{x_0}}(\lambda|_{B_{x_0}})$.
We denote the restriction of $\ev_{x_0}$ to $M^{N}(\lambda+\lambda_0)$ by same notation $\ev_{x_0}$.
Then, it suffices to show that $\ev_{x_0}$ is bijection
between $M^{N}(\lambda+\lambda_0)$ and $(M/\frakm({x_0})M)^{N_{x_0}}(\lambda|_{B_{x_0}})$.
\begin{align*}
\xymatrix{
M \ar[r]^{\ev_{x_0}}&M_{x_0}\\
M^{N}(\lambda+\lambda_0) \ar@{^{(}->}[u] \ar[r]^{\ev_{x_0}}&M_{x_0}^{N_{x_0}}(\lambda|_{B_{x_0}}) \ar@{^{(}->}[u]
}
\end{align*}

\textit{(injectivity)}.
Suppose $m \in M^{N}(\lambda+\lambda_0)$ and $\ev_{x_0}(m) = 0$.
Since $m$ is $B$-eigenvector, $m \in \frakm(b{x_0})M$ for any $b \in B$.
Then, we have $m \in \bigcap_{y \in B{x_0}}\frakm(y)M$.
Since $\bigcap_{y \in B{x_0}}\frakm(y)M = 0$ from Lemma \ref{lemma:stable_multiplicity_injective},
this implies $m = 0$. This shows $\ev_{x_0}$ is injective.

\textit{(surjectivity)}.
First, we show the surjectivity for the case that $M$ is a free $\rring{X}$-module.
Suppose $M \simeq \rring{X}\otimes W$ for some finite dimensional rational representation $W$ of $G$.
In this case, $\ev_{x_0}$ is actually the evaluation map at ${x_0}$.
Take $m \in W^{N_{x_0}}(\lambda|_{B_{x_0}})$, and for $b \in B$ put
\begin{align*}
\varphi(b{x_0})=b^{-\lambda-\lambda_0}(bm).
\end{align*}
Then, $\varphi$ is well-defined as an element of $\rring{B{x_0}}\otimes W$, and $\varphi$ is $B$-eigenvector of weight $\lambda + \lambda_0$.
From Lemma \ref{lemma:stable_multiplicity_surjective}, there exists a $B$-eigenvector $f_{\mu} \in \rring{X}$ such that $f_{\mu}\varphi \in \rring{X}\otimes W$.
$f_{\mu}\varphi$ is in $(\rring{X}\otimes W)^{N}(\lambda+\lambda_0+\mu)$.
By Theorem \ref{theorem:general stability}, the multiplication operator $f_{\mu}\cdot : (\rring{X}\otimes W)^{N}(\lambda+\lambda_0)\rightarrow (\rring{X}\otimes W)^{N}(\lambda+\lambda_0+\mu)$
is bijective.
Then, $\varphi$ is in  $(\rring{X}\otimes W)^{N}(\lambda+\lambda_0)$.
Since $\varphi({x_0}) = m$, $\ev_{x_0}$ is surjective.

Next, we show the surjectivity for general cases.
Since $M$ is finitely generated as a $\rring{X}$-module, there exists a finite dimensional $G$-subrepresentation $W \subset M$ such that
the $(\rring{X},G)$-homomorphism $\times : \rring{X}\otimes W \rightarrow M$ defined by the multiplication is surjective.
Then, we have the following commutative diagram:
\begin{align*}
\xymatrix{
\rring{X} \otimes W \ar[r]^{\ \ \ \ \ev_{x_0}} \ar[d]^{\times} & W \ar[d] \\
M \ar[r]^{\ev_{x_0}} & M_{x_0},
}
\end{align*}
and all arrows are surjective.
Take $\lambda'_0 \in \dweight(X)$ described in Theorem \ref{theorem:general stability} for $M = \rring{X} \otimes W$.
By restricting the above diagram to the subspace of $B$-eigenvectors of weight $\lambda+\lambda'_0$, we have
\begin{align*}
\xymatrix{
(\rring{X} \otimes W)^{N}(\lambda+\lambda'_0) \ar[r]^{\ \ \ \ \ \ \ \ev_{x_0}} \ar[d] & W^{N_{x_0}}(\lambda|_{B_{x_0}}) \ar[d] \\
M^{N}(\lambda+\lambda'_0) \ar[r]^{\ev_{x_0}} & M_{x_0}^{N_{x_0}}(\lambda|_{B_{x_0}}).
}
\end{align*}
Since $G$ and $L$ are reductive, the vertical arrows are surjective.
From the free module case, the above horizontal arrow is surjective.
Then, $\ev_{x_0}: M^{N}(\lambda+\lambda'_0) \rightarrow M_{x_0}^{N_{x_0}}(\lambda|_{B_{x_0}})$ is also surjective.

Since $\dim(M^{N}(\lambda+\lambda_0))\geq \dim(M^{N}(\lambda+\lambda'_0))$ by the result of Theorem \ref{theorem:general stability},
$\ev_{x_0}: M^{N}(\lambda+\lambda_0) \rightarrow M_{x_0}^{N_{x_0}}(\lambda|_{B_{x_0}})$ is also surjective.
\end{proof}

\begin{remark}
The injectivity is true in more general settings.
For example, suppose $X$ is a projective $G$-variety that has an open dense Borel orbit $Bx_0$, and $\pi: \calV \rightarrow X$ is a $G$-equivariant algebraic vector bundle over $X$.
Then, the global sections $\Gamma(X, \calV)$ and the evaluation map $\ev_{x_0}: \Gamma(X, \calV) \rightarrow \pi^{-1}({x_0})$ satisfy the injectivity
as in the above proof.
This implies that the multiplicity with respect to $G$ can be bounded by the multiplicity with respect to $L$ as in Theorem \ref{theorem:stable_multiplicity_M}.
See Section \ref{section:projective_example} for examples.
\end{remark}

We can remove the finiteness of $M$ if we admit that the conclusion becomes weaker.

\begin{corollary}\label{corollary:stable_multiplicity_M}
Let $M$ be a $(\rring{X}, G)$-module with no zero divisors.
Then, we have
\begin{align*}
\sup_{\mu \in \dweight(X)}\{m_M^{G}(\lambda+\mu)\}=m_{M_{x_0}}^{L}(\lambda|_{B_{x_0}}).
\end{align*}
for any $\lambda \in \dweight(M)$.
\end{corollary}

\begin{proof}
Take a weight $\lambda \in \dweight(M)$.
By the same proof as the injectivity in Theorem \ref{theorem:stable_multiplicity_M},
we have
\begin{align*}
\sup_{\mu \in \dweight(X)}\{m_M^{G}(\lambda+\mu)\} \leq m_{M_{x_0}}^{L}(\lambda|_{B_{x_0}}).
\end{align*}
For any finitely dimensional $L$-subrepresentation $\overline{N} \subset M_{x_0}$,
we can take a finitely generated $(\rring{X}, G)$-submodule $N$ such that
$N/(N\cap \frakm({x_0})M) \supset \overline{N}$.
If necessary, we can take $N$ such that $m_N^{G}(\lambda) \neq 0$.
Since the canonical map $N/\frakm({x_0})N \rightarrow N/(N\cap \frakm({x_0})M)$ is surjective,
the converse inequality is followed from Theorem \ref{theorem:stable_multiplicity_M}.
\end{proof}

For a rational representation $V$ of $G$, we denote by $C_G(V)$ the supremum of $m_V^{G}$.

\begin{corollary}\label{corollary:upper_bound}
Let $M$ be a $(\rring{X}, G)$-module with no zero divisors.
Then, the following equation holds:
\begin{align*}
C_G(M)=C_L(M_{x_0}).
\end{align*}
Especially, $M$ is multiplicity-free as a representation of $G$ if and only if $M_{x_0}$ is multiplicity-free as a representation of $L$.
\end{corollary}

\begin{proof}
By Theorem \ref{corollary:stable_multiplicity_M}, $C_G(M) \leq C_L(M_{x_0})$ is clear.
It suffices to show that any character $\lambda$ of $B_{x_0}$ such that $m^L_M(\lambda) \neq 0$ can be extended
to an character $\overline{\lambda}$ of $T$ such that $m^G_M(\overline{\lambda}) \neq 0$.
As in the proof of the surjectivity in Theorem \ref{theorem:stable_multiplicity_M},
we can assume that $M$ is free $\rring{X}$-module of finite rank, $\rring{X} \otimes W$.
We take a character $\lambda$ of $B_{x_0}$ such that $m^L_{M_{x_0}}(\lambda) \neq 0$, and take $m \in W^{N_{x_0}}(\lambda)$.
There exists a character $\lambda'$ of $T$ such that $\lambda'|_{B_{x_0}}=\lambda$.
For $\varphi(b{x_0})=b^{-\lambda'}(bm)$, we can find $f_\mu \in \rring{X}^{N}(\mu)$ such that $f_\mu \varphi \in (\rring{X}\otimes W)^N(\lambda'+\mu)$.
Since $(\lambda'+\mu)|_{B_{x_0}}=\lambda'|_{B_{x_0}}=\lambda$, $\overline{\lambda}:=\lambda'+\mu$ satisfies the desired conditions.
This completes the proof.
\end{proof}

\section{Examples of stability theorems}\label{section:examples}

In this section, we will show some stability theorems for some examples.

\subsection{Stability theorem for quasi-affine spherical homogeneous spaces}
Let $G$ be a connected reductive algebraic group and $H$ be a closed subgroup of $G$.
We assume that $(G, H)$ is a spherical pair (i.e. there exists a Borel subgroup of $G$ such that $BH$ is open dense in $G$),
and assume that $G/H$ is a quasi-affine variety.
Put $L := \{g \in H: gBH \subset BH\}$.

For a finite dimensional rational representation $W$ of $H$, we define the \define{induced representation} of $W$ by
\begin{align*}
\Ind^G_H(W) := (\rring{G}\otimes W)^{H}.
\end{align*}
$\Ind^G_H(W)$ becomes a $(\rring{G}^{H}, G)$-module via the left $G$-action and the multiplication of $\rring{G}^{H}$.
Here, the $H$-invariant part is taken via its right action on $\rring{G}$.

The following fact is known as the characterization of quasi-affine homogeneous spaces.
(see \cite[Theorem 3.12]{Ti11})
\begin{proposition}\label{proposition:quasi-affine}
Let $G$ be a linear algebraic group, and $H$ be a closed subgroup of $G$.
Then, the following three conditions are equivalent.
\begin{enumerate}[i)]
	\item $G/H$ is quasi-affine.
	\item the quotient field of $\rring{G/H}$ is equal to the rational function field of $G/H$.
	\item For any $H$-representation $W$, there exists a finite dimensional representation $V$ of $G$
	such that $W$ can be embedded in $V$ as a representation of $H$.
\end{enumerate}
\end{proposition}

Applying Theorem \ref{theorem:general stability} and Theorem \ref{theorem:stable_multiplicity_M}
to $X = G/H$ and $M = \Ind^G_H(W)$, we have the following theorem.
In the case that $H$ is semisimple, this theorem is due to F. Sat{\=o} in \cite{Sa93}.

\begin{theorem}\label{theorem:stable_theorem_for_homogeneous_space}
Let $W$ be a finite dimensional rational representation of $H$.
Then, there exists a weight $\lambda_0 \in \dweight(G/H)$ such that
\begin{align*}
m^G_{\Ind^G_H W}(\lambda+\lambda_0) = m^L_W(\lambda|_{B \cap H})
\end{align*}
for any $\lambda \in \dweight(\Ind^G_H(W))$.
\end{theorem}

\begin{proof}

First, we show that $X = G/H$ and $M = \Ind^G_H (W)$ satisfy the conditions of Theorem \ref{theorem:general stability} and Theorem \ref{theorem:stable_multiplicity_M}.
By the assumption that $X = G/H$ is a quasi-affine $G$-variety, the quotient field of $\rring{X}$ coincides with the rational function field of $X$.
From Proposition \ref{proposition:quasi-affine}, $W$ can be embedded in a finite dimensional representation $V$ of $G$ as a representation of $H$.
The embedding $W \hookrightarrow V$ induces the injection as a $(\rring{G/H}, G)$-module:
\begin{align*}
\Ind^G_H(W) \hookrightarrow \rring{G/H}\otimes V.
\end{align*}
Since $\rring{G/H}\otimes V$ is Noetherian $\rring{G/H}$-module,
$\Ind^G_H(W)$ is Noetherian, and then finitely generated.
From the inclusion, it is obvious that $\rring{G/H}$ has no zero divisors in $\Ind^G_H(W)$.
All conditions are verified.

Next, we show that $M/\frakm(eH)M \simeq W$ as a representation of $L$.
We can identify $\Ind^G_H(W)$ with the set of global sections $\Gamma(G/H, G \times_H W)$ of a vector bundle $G\times_H W \rightarrow G/H$.
Since $G/H$ is a quasi-affine variety, the sheaf constructed from the vector bundle corresponds to the sheaf constructed from $\Ind^G_H(W)$.
This shows that $(\Ind^G_H(W))_{eH} \simeq W$ as a representation of $G_{eH}$.
\end{proof}

\subsection{Some examples for projective varieties}\label{section:projective_example}
In this section, we treat flag varieties.
Let $G$ be a connected reductive algebraic group,
and $P$ be a parabolic subgroup of $G$.
Take a close subgroup $H$ of $G$ such that $G/P$ is a spherical $H$-variety.
Note that if $H$ is a Levi subgroup of $G$, such pairs $(G, H, P)$ were classified in \cite{Ko07_2, Ta12}.

Fix a Borel subgroup $B$ of $H$.
Since $G/P$ is a spherical $H$-variety, there exists a point ${x_0} \in G$ such that $B{x_0}P$ is open dense in $G$.
Put $L := \{g \in H_{{x_0}P}: gB{x_0}P \subset B{x_0}P\}$.
The same result as Theorem \ref{theorem:stable_theorem_for_homogeneous_space} is not true for $G/P$
since $G/P$ is projective (see Example \ref{example:projective}).
However, we have the following weaker theorem.
\begin{theorem}\label{theorem:projective}
Let $W$ be a irreducible representation of $P$.
Then, there exists a character $\lambda_0$ of $P$ such that
\begin{align*}
C_H(\Ind_P^G(W \otimes \CC_{\lambda_0+\lambda}))=C_L(W)
\end{align*}
for any character $\lambda$ of $P$ satisfying $\Ind^G_P(\CC_{\lambda}) \neq 0$.
Here, $W$ is considered as a representation of $L$ by the inclusion ${x_0}^{-1}L{x_0} \subset P$.
\end{theorem}

Fix a Levi decomposition $P=QN$, where $N$ is a unipotent radical of $P$.
Put $P':=[Q,Q]N$ and $A:=Q/[Q,Q]$.
By Proposition \ref{proposition:quasi-affine}, $G/P'$ is a quasi-affine spherical $H\times A$-variety.
The action of $A$ on $G/P'$ is given by
\begin{align*}
h\cdot gP'=gh^{-1}P'
\end{align*}
for $h \in A$ and $g \in G$.
Note that $B\times A$ is a Borel subgroup of $H \times A$.
For the proof of Theorem \ref{theorem:projective}, we show the following lemma.

\begin{lemma}\label{lemma:projective_L=L'}
We set
\begin{align*}
L' := \{(g,h) \in H \times A: (g,h)\cdot {x_0}P'={x_0}P', (g,h)\cdot B{x_0}P \subset B{x_0}P\}.
\end{align*}
Then, there exists a homomorphism $\varphi:L\rightarrow A$ such that
\begin{align}
L'=\{(g,\varphi(g)) \in L\times A: g \in L\}. \label{equation:L_and_L'}.
\end{align}
\end{lemma}

\begin{proof}
First, we define the homomorphism $\varphi$.
Take $g \in L$.
By definition, we have $gB{x_0}P \subset B{x_0}P$ and $g{x_0}P={x_0}P$.
From $g{x_0}P'\subset {x_0}P = \sqcup_{l\in A}{x_0}lP'$, there exists a unique element $\varphi(g) \in A$
such that $g{x_0}P'=x\varphi(g)P'$.
It is obvious that $\varphi$ is a homomorphism from $L$ to $A$.

Next, we show that $\varphi$ satisfies the condition.
By the definition of $L$ and $L'$, we have $(g, \varphi(g)) \in L'$ for any $g \in L$.
For the converse inclusion, we take $(g, h) \in L'$.
Since $(g, h) \in L'$, we have $g{x_0}h^{-1}P'={x_0}P'$ and $gB{x_0}P \subset B{x_0}P$.
This implies that $g \in L$.
Since ${x_0}P'=g{x_0}h^{-1}P'={x_0}\varphi(g)h^{-1}P'$, we have $\varphi(g)=h$.
This shows the lemma.
\end{proof}

\begin{proof}[Proof of Theorem \ref{theorem:projective}]
We apply Corollary \ref{corollary:upper_bound} to $X=G/P'$ and $M=\Ind^G_{P'}(W)$.
Here, we replace $G$ in the corollary by $H \times A$,
and then $L$ in the corollary is equal to $L'$ in the above lemma.

We will determine the action of $L'$ on $M/\frakm({x_0}P')M$.
Note that $M/\frakm({x_0}P')M$ is isomorphic to $W$ as a $\CC$-vector space.
Take $(g, \varphi(g)) \in L'$.
For $f \in (\rring{G}\otimes W)^{P'}$, we have
\begin{align*}
((g, \varphi(g))\cdot f)({x_0})&=\varphi(g)f(g^{-1}{x_0}\varphi(g)) \\
&=\varphi(g)f({x_0} {x_0}^{-1}g^{-1}{x_0}\varphi(g)) \\
&=\varphi(g)({x_0}^{-1}g^{-1}{x_0}\varphi(g))^{-1}f({x_0})\\
&={x_0}^{-1} g {x_0} f({x_0}).
\end{align*}
Then, the action of $L'$ on $M/\frakm({x_0}P')M \simeq W$ coincides with the action of $L$.
Therefore, we have $C_L'(W)=C_L(W)$.

From Corollary \ref{corollary:upper_bound}, there exists $\lambda' \in \dweight_{H\times A}(\Ind^G_{P'}(W))$ such that
\begin{align}
m^{H\times A}_{\Ind^G_{P'}(W)}(\lambda')=C_{L'}(W).
\end{align}
We write $\lambda'=-\lambda_0+\lambda_1$,
where $\lambda_0$ is a character of $P$ and $\lambda_1$ is a character of $B$.

Let us show that $\lambda_0$ satisfies the desired condition.
There exist the following isomorphisms of representations of $H$:
\begin{align*}
\Ind^G_{P'}(W)(-\lambda_0) &\simeq (\Ind^G_{P'}(W) \otimes \CC_{\lambda_0})^{A}\\
&\simeq ((\rring{G} \otimes W)^{P'} \otimes \CC_{\lambda_0})^{A}\\
&\simeq (\rring{G} \otimes W \otimes \CC_{\lambda_0})^{P}\\
&\simeq \Ind^G_P(W\otimes \CC_{\lambda_0}).
\end{align*}
Then, we have $C_H(\Ind^G_P(W\otimes \CC_{\lambda_0}))=C_L(W)$.
Again, from the above isomorphisms for $W=\CC$, $\Ind^G_{P}(\CC_\lambda)$ is nonzero if and only if
there exists a character $\nu$ of $B$ such that $-\lambda+\nu \in \dweight_{H\times A}(G/P')$.
Combining this with Theorem \ref{theorem:stable_multiplicity_M}, the proof is completed.
\end{proof}

We introduce an example that $\Ind^G_P(W)$ is nonzero and we can not take $\lambda_0=0$.

\begin{example}\label{example:projective}
Let $G=\GL(8,\CC), H=\GL(4,\CC)\times \GL(4,\CC)$.
$H$ is block diagonal in $G$.
Let $P$ be a maximal parabolic subgroup of $G$ containing $H$ and all lower triangular matrices,
and $B$ be a Borel subgroup of $H$ containing all upper triangular matrices in $H$.
We take a point
\begin{align*}
{x_0}:=
\begin{pmatrix}
I & J\\
0 & I
\end{pmatrix},
\end{align*}
where $J$ is a anti-diagonal matrix with any anti-diagonal entries $1$.
Then, $B{x_0}P$ is open dense in $G$. (This is essential same as the case of Hermitian symmetric spaces.)
In this case, $L$ is of the following form:
\begin{align*}
L=\left\{
\left(
\begin{pmatrix}
a & 0 \\
0 & b
\end{pmatrix},
\begin{pmatrix}
b & 0 \\
0 & a
\end{pmatrix}\right):
a, b \in \CC^{\times}
\right\}.
\end{align*}
Note that $L$ commutes with ${x_0}$.

We consider a representation $W=S^2(\bigwedge^2(\CC^4))/\bigwedge^4(\CC^4)$ of $H$,
where the first factor of $\GL(4,\CC)\times \GL(4,\CC)$ acts on $W$ in standard way
and the second factor acts on $W$ trivially.
$W$ is a irreducible representation of $H$ with highest weight $(2,2,0,\ldots, 0)$ in standard coordinates.
We extend the representation $W$ to $P$ by letting the unipotent radical of $P$ act trivially.
Then, the induced representation $\Ind^G_P(W)$ is a irreducible representation of $G$ with highest weight $(2,2,0,\ldots, 0)$.
By using Littlewood--Richardson rule, $\Ind^G_P(W)|_H$ is multiplicity-free, and then $C_H(\Ind^G_P(W))=1$.
However, $W|_L$ is not multiplicity-free.
In fact, we can take two weight vectors with same weight $(1, 1, 1, 1, 0, 0, 0, 0)$ such as
\begin{align*}
e_1\wedge e_2 \cdot e_3 \wedge e_4, e_1\wedge e_3 \cdot e_2 \wedge e_4.
\end{align*}
Then, we have $C_H(\Ind^G_P(W))=1<2=C_L(W)$.
\end{example}

\subsection{Stability theorem for highest weight modules}

Here, we will show the stability theorem for unitary highest weight modules.
Let $G$ be a connected simple real Lie group of Hermitian type with finite center.
Fix a positive root system $\Delta^{+}$ and strongly orthogonal roots $\{\sor_1, \sor_2, \ldots, \sor_r \}$ as in Section \ref{section:strongly_orthogonal}.

Before we state our theorem, we prepare some lemmas relevant to $K_{\CC}$-orbit $\calO_i$.
Fix $1\leq m \leq r$.
We set $\fraka_m:=\bigoplus_{i=1}^{m} \RR (X_{\sor_i}+\overline{X_{\sor_i}})$, $\frakt_m:=\bigoplus_{i=1}^{m} \CC [X_{\sor_i}, \overline{X_{\sor_i}}]$
, and $L = Z_{K_{\CC}}(\fraka_m)$.
We will show that $L$ satisfies the conditions L-1) $\sim$ L-4) in Proposition \ref{proposition:L-subgroup}.

\begin{lemma}\label{lemma:B-orbit}
Let $B$ be a Borel subgroup of $K_{\CC}$ determined by the positive system $\Delta_c^{+}$.
Then, $\Ad(B)X_m$ is open dense in $\overline{\calO_m}$.
\end{lemma}

\begin{proof}
Since $\Ad(K_{\CC})X_m$ is open dense in $\overline{\calO_m}$, it suffices to show that $[\frakk_{\CC},X_m] = [\frakb, X_m]$.
From Proposition \ref{proposition:restricted_root}, we have 
\begin{align*}
[\frakg_{-\gamma}, X_m] &= 
\begin{cases}
0 & (m < j) \\
\frakg_{-\gamma+\gamma_j} & (1 \leq j \leq m)
\end{cases}
&\text{ for any }& \gamma \in C_{ij}, \\
[\frakg_{-\gamma}, X_m] &= 0 &\text{ for any }& \gamma \in C_i \cup C_0,
\end{align*}
for any $i, j (1 \leq i < j \leq r)$.
Since $-\gamma+\gamma_j \in P_{ij}$ for any $\gamma \in C_{ij}$, there exists a $\gamma' \in C_{ij}$ such that $\gamma' + \gamma_i = -\gamma + \gamma_j$.
Then, we have $[\frakb, X_m] \supset [\overline{\frakb}, X_m]$, where $\overline{\frakb}$ is the opposite Borel subalgebra (or equivalently complex conjugate) of $\frakb$.
This implies $[\frakk_{\CC}, X_m] = [\frakb, X_m]$.
\end{proof}

\begin{lemma}\label{lemma:B-structure}
Let $B$ be the same as the above lemma.
Then, $B_{X_m}$ has the semi-direct product decomposition: $B_{X_m} = (T_{\CC})_{X_m}N_{X_m}$.
Here, $N$ is a unipotent radical of $B$.
\end{lemma}

\begin{proof}
$B_{X_m} \supset (T_{\CC})_{X_m}N_{X_m}$ is obvious.

For converse inclusion, we take $b \in B_{X_m}$, and write $b = tn$ for $t \in T_{\CC}$ and $n \in N$.
By Kostant-Hua-Schmid Theorem (Proposition \ref{proposition:kostant_hua_schmid}), $B_{X_m}$ is contained in $\bigcap_{i=1}^{m}\ker \sor_i$.
Then, we have $t \in \bigcap_{i=1}^{m}\ker \sor_i|_{T_{\CC}}$.
Considering the $T_{\CC}$-action on $X_m$, $\bigcap_{i=1}^{m}\ker \sor_i|_{T_{\CC}}$ is equal to $(T_{\CC})_{X_m}$.
Therefore, we have $t \in (T_{\CC})_{X_m}$ and then $n \in N_{X_m}$.
This shows the lemma.
\end{proof}

\begin{lemma}\label{lemma:L_is_L_subgroup}
$L$ satisfies the conditions L-1) $\sim$ L-4) in Proposition \ref{proposition:L-subgroup}.
\end{lemma}

\begin{proof}
By definition, the condition L-1) is clear.

First, we compute the triangular decomposition of $\frakl$.
For any $g \in L$ and $i\ (1 \leq i \leq m)$, we have
\begin{align*}
X_{\sor_i}+\overline{X_{\sor_i}} &= \Ad(g)(X_{\sor_i}+\overline{X_{\sor_i}}) \\
&= \Ad(g)(X_{\sor_i}) + \Ad(g)(\overline{X_{\sor_i}}).
\end{align*}
Since $\Ad(g)(X_{\sor_i}) \in \frakp_{+}$ and $\Ad(g)(\overline{X_{\sor_i}}) \in \frakp_{-}$, $g$ stabilizes $X_{\sor_i}$ and $\overline{X_{\sor_i}}$.
This imply that
\begin{align}
L &= Z_{K_{\CC}}\left(\bigoplus_{i=1}^{m}(\frakg_{\sor_i} \oplus \frakg_{-\sor_i})\right) \label{equation:L'}\\
\frakl &= Z_{\frakk_{\CC}}\left(\bigoplus_{i=1}^{m}(\frakg_{\sor_i} \oplus \frakg_{-\sor_i})\right). \nonumber
\end{align}
Since the right hand side is stable under $\ad(\frakt_{\CC})$-action, so is $\frakl$.
Then, we have
\begin{align}
\frakl = \overline{(\frakl \cap \frakn)}\oplus (\frakl \cap \frakt_{\CC}) \oplus (\frakl \cap \frakn).\label{equation:tri_L}
\end{align}

We will show that the Lie algebra $\frakb_{X_m}$ of $B_{X_m}$ is a Borel subalgebra of $\frakl$.
By Proposition \ref{proposition:restricted_root}, for any $i, j (1 \leq i < j \leq r)$ 
\begin{align*}
[\frakg_{\gamma}, X_m] &= 
\begin{cases}
0 & (m < i) \\
\frakg_{\gamma+\gamma_i} & (1 \leq i \leq m)
\end{cases}
&\text{ for any }& \gamma \in C_{ij} \cup C_{i}, \\
[\frakg_{\gamma}, X_m] &= 0 &\text{ for any }& \gamma \in C_0.
\end{align*}
This implies that $\frakb_{X_m}$ has the following decomposition:
\begin{align*}
\frakb_{X_m} = \frakt_m^{\perp} \oplus 
\bigoplus_{\substack{\gamma \in C_{ij}\cup C_{i} \\ m < i < j \leq r}}\frakg_{\gamma} \oplus \bigoplus_{\gamma \in C_0}\frakg_{\gamma},
\end{align*}
where $\frakt_m^\perp$ is the orthogonal complement of $\frakt_m$ in $\frakt_{\CC}$ with respect to the Killing form.
Then, we have $\frakb_{X_m} = Z_{\frakb}(\bigoplus_{i = 1}^{m}(\frakg_{\sor_i}\oplus \frakg_{-\sor_i})) = \frakl \cap \frakb$.
From (\ref{equation:tri_L}), this shows the condition L-4).

We can show that $L=Z_{L}(\frakt_m^\perp)L_0$ by the same proof as \cite[Proposition 7.49]{Kn02}.
By the equation (\ref{equation:L'}), $Z_{L}(\frakt_m^\perp)$ is contained in $Z_{K_{\CC}}(\frakt_{\CC})\cap L \subset (T_{\CC})_{X_m}$.
This shows that $(T_{\CC})_{X_m}$ meets every connected components of $L$, and so does $B_{X_m}$.
Then, we have shown the condition L-3).

By the proof in Lemma \ref{lemma:B-structure}, $(T_{\CC})_{X_m}$ is equal to $\bigcap_{i=1}^{m}\ker \sor_i|_{T_{\CC}}$.
This implies that $(T_{\CC})_{X_m}$ is contained in $L$.
Since $B_{X_m} = (T_{\CC})_{X_m}N_{X_m}$ and $N_{X_m}$ is connected, $B_{X_m}$ is contained in $L$.
This completes the proof.

\end{proof}

Suppose $\calH$ be a highest weight module of $G$.
We consider the $(\frakg, K)$-module $\calH$ as a $(\rring{\AV(\calH)}, K_{\CC})$-module.
By Proposition \ref{proposition:joseph}, $\rring{\AV(\calH)}$ has no zero divisors in $\calH$.
Then, applying Theorem \ref{theorem:general stability} and \ref{theorem:stable_multiplicity_M} to $\calH$, we have the following theorem.

\begin{theorem}\label{theorem:stable_theorem_for_highest_weight_module_maximal_compact}
Let $\calH$ be a unitary highest weight module of $G$ with the associated variety $\overline{\calO_m}$.
Then, there exists a $\lambda_0 \in \dweight(\overline{\calO_m}) = \{-\sum_{i=1}^m c_i \sor_i: c_1 \geq c_2 \geq \cdots \geq c_m \geq 0, c_i \in \ZZ \}$ such that
\begin{align*}
m^{K_{\CC}}_{\calH}(\lambda+\lambda_0) = m^{L}_{\calH_{X_m}}(\lambda|_{T_{X_m}})
\end{align*}
for any $\lambda \in \dweight(\calH)$.
\end{theorem}

\begin{proof}
By Kostant-Hua-Schmid Theorem (Proposition \ref{proposition:kostant_hua_schmid}), $\overline{\calO_m}$ is a spherical affine $K_{\CC}$-variety,
and $\dweight(\overline{\calO_m}) = \{-\sum_{i=1}^m c_i \sor_i: c_1 \geq c_2 \geq \cdots \geq c_m \geq 0, c_i \in \ZZ \}$.
Since $\calH$ is generated by $\calH^{\frakp^{+}}$ as a $\rring{\overline{\calO_m}}$-module, $\calH$ is a finitely generated $(\rring{\overline{\calO_m}}, K_{\CC})$-module.
We have shown that $L$ satisfies the conditions L-1) $\sim$ L-4) in Lemma \ref{lemma:L_is_L_subgroup}.
Then, this completes the proof.
\end{proof}

The following corollary is direct consequence of Corollary \ref{corollary:upper_bound} and 
Theorem \ref{theorem:stable_theorem_for_highest_weight_module_maximal_compact}.

\begin{corollary}\label{corollary:upper_bound_highest}
Let $\calH$ be a unitary highest weight module of $G$ with the associated variety $\overline{\calO_m}$.
Then, we have
\begin{align*}
C_K(\calH)=C_L(\calH_{X_m}).
\end{align*}
Moreover, $\calH|_K$ is multiplicity-free if and only if $\calH_{X_m}|_L$ is multiplicity-free.
\end{corollary}

\begin{remark}
In the proof of Theorem \ref{theorem:stable_theorem_for_highest_weight_module_maximal_compact}
and Corollary \ref{corollary:upper_bound_highest},
we did not use the assumption that $\calH$ is irreducible.
Then, we can apply the theorem in the assumption that $\calH$ is the finite direct sum of unitary highest weight modules
with same associated varieties.
\end{remark}

\subsection{Stability theorem for holomorphic symmetric pairs}\label{section:stability_discrete}

Let $G$ be a connected simple real Lie group of Hermitian type with finite center, 
and $\tau$ be an involutive automorphism of $G$ commuting with a Cartan involution $\theta$ of $G$.
We put $H = (G^{\tau})_0$, the identity component of the fixed point group of $\tau$.
We assume that $(\frakg, \frakh)$ is a holomorphic symmetric pair.
Suppose $\calH$ is a holomorphic discrete series representation of $G$.

Before we state the theorem, we set up some notations.
We fix a Cartan subalgebra $\frakt^{\tau}$ of $\frakk^{\tau}$, and
fix a positive system $\Delta^{+}(\frakg^{\tau\theta}_{\CC}, \frakt^{\tau}_{\CC})$ such that
$\Delta(\frakp^{-\tau}_{+}, \frakt^{\tau}_{\CC}) \subset \Delta^{+}(\frakg^{\tau\theta}_{\CC}, \frakt^{\tau}_{\CC})$.
Let $B=TN$ be a Borel subgroup of $(H \cap K)_{\CC}$ corresponding to the positive system $\Delta^{+}(\frakg^{\tau\theta}_{\CC}, \frakt^{\tau}_{\CC})$.

We will take strongly orthogonal roots $\{\sor_1, \sor_2, \ldots, \sor_r \}$ in $\Delta(\frakp^{-\tau}_{+}, \frakt^{\tau}_{\CC})$
by similar way in Section \ref{section:strongly_orthogonal}.
However, it may be possible that $\frakg^{\tau\theta}$ is not a simple Lie algebra.
Suppose $\frakg^{\tau\theta}=\bigoplus_{i = 1}^{n} \frakh_i$.
We set up an lexicographical order on $\Delta(\frakg^{\tau\theta}, \frakt^{\tau}_{\CC})$ such that
any elements of $\Delta^{+}(\frakg^{\tau\theta}_{\CC}, \frakt^{\tau}_{\CC})$ are positive
and $\Delta^{+}(\frakh_i, \frakt^{\tau}_{\CC}) < \Delta^{+}(\frakh_j, \frakt^{\tau}_{\CC})$ for any $i, j\ (i < j)$.
Here, we write $X < Y$ if $x < y$ for any $x \in X$ and $y \in Y$.
Replacing the term 'lowest root' by 'minimum root' in the definition of Section \ref{section:strongly_orthogonal},
we can take strongly orthogonal roots $\{\sor_1, \sor_2, \ldots, \sor_r \}$ in $\Delta(\frakp^{-\tau}_{+}, \frakt^{\tau}_{\CC})$.

Put $\fraka = \bigoplus_{i=1}^{r} \RR (X_{\sor_i}+\overline{X_{\sor_i}}) \subset \frakp^{-\tau}$.
Then, $\fraka$ is a maximal abelian subspace of $\frakp^{-\tau}$, and $r = \dim_{\RR}(\fraka) = \RR$-$\rank(\frakg^{\tau\theta})$.

\begin{theorem}\label{theorem:stable_theorem_for_holomorphic_discrete_series}
Let $\calH$ be a holomorphic discrete series representation of $G$.
We put $L = Z_{H \cap K}(\fraka)$.
Then, there exists a $\lambda_0 \in \dweight(\frakp_{+}^{-\tau})$ such that
\begin{align*}
m^{H}_{\calH}(\lambda+\lambda_0) = m^{L}_{\calH^{\frakp_{+}}}(\lambda|_{T_x})
\end{align*}
for any $\lambda \in \dweight(\calH)$.
Here, we denote by $m^{H}_{\calH}(\lambda)$ the multiplicity of the unitary highest weight module with highest weight $\lambda$
with respect to $\frakp^{\tau}_{+}\oplus \frakb$.
\end{theorem}

\begin{proof}
By Proposition \ref{proposition:reduction_to_compact},
the decomposition of $\calH|_{H}$ is reduced to
the decomposition of $(S(\frakp_{-}^{-\tau}) \otimes \calH^{\frakp_{+}})|_{H \cap K}$.
There exists a canonical isomorphism as $H \cap K$-representation:
\begin{align*}
S(\frakp_{-}^{-\tau}) \otimes \calH^{\frakp_{+}} &\simeq N^{\frakg^{\tau\theta}}(\calH^{\frakp_{+}}) \\
&\simeq \calU(\frakg^{\tau\theta})\calH^{\frakp_{+}} (\subset \calH).
\end{align*}
This implies that $S(\frakp_{-}^{-\tau}) \otimes \calH^{\frakp_{+}}$ is isomorphic to the finite direct sum of
some holomorphic discrete series representations of $(G^{\theta\tau})_0$ as $H\cap K$-representation.
Then, applying Theorem \ref{theorem:stable_theorem_for_highest_weight_module_maximal_compact} to $N^{\frakg^{\tau\theta}}(\calH^{\frakp_{+}})$,
we prove the theorem.
\end{proof}

\begin{corollary}\label{corollary:upper_bound_holomorphic_discrete_series}
Let $\calH$ be a holomorphic discrete series representation of $G$.
We put $L = Z_{H \cap K}(\fraka)$.
Then, we have $C_H(\calH) = C_L(\calH^{\frakp_{+}})$.
Moreover, $\calH|_H$ is multiplicity-free if and only if $\calH^{\frakp_{+}}|_L$ is multiplicity-free.
\end{corollary}

\section{Branching laws and $\epsilon$-family}\label{section:upper_bound}

In this section, we discuss the relation between branching laws and $\epsilon$-family.

\subsection{Motivation}

Our main problem is the following:
\begin{question}\label{question:motivation}
Let $H, H'$ be reductive subgroups of a real reductive Lie group $G$, and $\pi$ be an irreducible unitary representation of $G$.
We assume that their complexifications are conjugate by an inner automorphism
(i.e. there exists a $g \in G_{\CC}$ such that $g H_{\CC} g^{-1} = H'_{\CC}$).
Then, what properties of branching laws of $\pi|_{H}$ and $\pi|_{H'}$ are preserved?
\end{question}
In simpler terms, it says whether there exists a method analogous to Weyl's unitary trick for branching problems.

The following fact about upper bounds of multiplicities is known (see \cite[Theorem 4.3.]{PeSe12}).
\begin{fact}\label{fact:penkov}
Let $\frakg$ be a complex Lie algebra, and $\frakh$ be a reductive complex Lie subalgebra in $\frakg$.
Suppose $M$ and $N$ are irreducible representations of $\frakg_{\CC}$,
and the action of $\frakh$ on $M$ and $N$ is locally finite and completely decomposable.
If $\Ann_{\calU(\frakg_{\CC})}(M)=\Ann_{\calU(\frakg_{\CC})}(M)$, then we have $C_\frakh(M)=C_\frakh(N)$.
\end{fact}

We rewrite this fact in terms of our settings.
\begin{fact}
Let $H, H'$ be compact subgroups of a connected real reductive Lie group $G$, and $\pi$ be an irreducible unitary representation of $G$.
Suppose $H_{\CC}, H'_{\CC}$ are conjugate by an inner automorphism of $G_{\CC}$.
Then, we have $C_H(\pi) = C_H'(\pi)$.
\end{fact}

Thus, the suprema of the multiplicities are preserved if $H$ and $H'$ are compact subgroup of $G$.
This is an answer of the question \ref{question:motivation} in a special aspect.
Note that if $H, H'$ are conjugate by an inner automorphism of $G$, this fact is trivial.

On the other hand, if $(G, H, H')=(\Sp(n,\RR), \U(n), \GL(n,\RR))$ and $\pi$ is a unitary highest weight module of $\Sp(n,\RR)$,
$\pi|_{\GL(n\RR)}$ has continuous spectra although $\pi|_{\U(n)}$ has no continuous spectra.
This example implies that discrete decomposability is not preserved.

Our main result in this section is to show that
the branching laws for holomorphic discrete series representations with respect to two similar holomorphic symmetric pairs
are similar (see Theorem \ref{theorem:upper_bound_hol}).

At the end, we state an interesting example.
\begin{example}
Let $(G,H)$ be $(\SO(2p,2q),\U(p_1)\times \U(p_2) \times \U(q))$ and $\pi$ be a minimal representation of $\SO(2p, 2q)$,
and $(G',H')$ be $(\upO^{*}(2(p+q)),\U(p_1)\times \U(p_2) \times \U(q))$ and $\pi'$ be a minimal representation of $\upO^{*}(2(p+q))$
with $p_1+p_2=p$.
Then, their complexifications are $(\upO(2(p+q),\CC),\GL(p_1, \CC) \times \GL(p_2, \CC) \times \GL(q, \CC))$.
The annihilators of $\pi$ and $\pi'$ in $\calU(\frako(2(p+q),\CC))$ are same ideal called the Joseph ideal $\calJ$,
although the representations $\pi$ and $\pi'$ are not equivalent as $\frakg_{\CC}$-modules.
Here, we write same notation $\pi$ for Harish-Chandra module of $\pi$.
\begin{align*}
\pi(\calU(\frakg)^{H}) &\simeq (\calU(\frakg)/\calJ)^{H} \\
&\simeq (\calU(\frakg)/\calJ)^{H'}\\
&\simeq \pi'(\calU(\frakg)^{H'})
\end{align*}
In \cite{Mo08}, M. Moriwaki showed that $\pi|_{H}$ is multiplicity-free.
Then, $\pi'|_{H'}$ is also multiplicity-free by Fact \ref{fact:penkov}
(we can obtain this by a straightforward computation using Howe duality).
\end{example}

\subsection{$\epsilon$-family}\label{section:epsilon_family}
The following definition of $\epsilon$-family is due to T. {\=O}shima and J. Sekiguchi \cite{OsSe84}.
Let $\frakg$ be a real semisimple Lie algebra, and $\tau$ be an involutive automorphism commuting with a Cartan involution $\theta$ of $\frakg$.
Take a maximal abelian subspace $\fraka$ of $\frakp^{-\tau}$.
We put $\frakg(\fraka;\lambda):=\{X \in \frakg: [H,X]=\lambda(H)X \text{ for any }H \in \fraka \}$
and $\Sigma(\fraka):=\{\lambda \in \fraka^{*} \backslash \{0\}: \frakg(\fraka;\lambda) \neq 0\}$.
By Rossmann (see \cite{Ro79}), $\Sigma(\fraka)$ becomes a root system.

We will say a map $\epsilon : \Sigma(\fraka)\cup \{0\} \rightarrow \{1,-1\}$ is a \define{signature} of $\Sigma(\fraka)$
if $\epsilon(\alpha+\beta) = \epsilon(\alpha) \epsilon(\beta)$ for any $\alpha, \beta \in \Sigma(\fraka)\cup \{0\}$.
For a signature $\epsilon$, we define an involutive automorphism $\tau_{\epsilon}$ of $\frakg$ as follows:
\begin{align*}
\tau_{\epsilon}(X)=\epsilon(\alpha)\tau(X) \text{ for } X \in \frakg(\fraka; \alpha), \alpha \in \Sigma(\fraka) \cup \{0\}.
\end{align*}
We set $F((\frakg, \frakg^{\tau})):=\{(\frakg, \frakg^{\tau_{\epsilon}}): \epsilon \text{ is a signature of } \Sigma(\fraka)\}$,
and call it an \define{$\epsilon$-family of symmetric pairs}.
If $\tau = \theta$, we call $F((\frakg, \frakk))$ a \define{$\frakk_\epsilon$-family of symmetric pairs}.

We have the following proposition.
\begin{proposition}\label{proposition:epsilon-family}
Let $\tau$ be an involutive automorphism of $\frakg$.
Then, for any signature $\epsilon$ of $\Sigma(\fraka)$, 
\begin{enumerate}[i)]
	\item $(\frakg_{\CC}, \frakg_{\CC}^{\tau})$ and $(\frakg_{\CC}, \frakg_{\CC}^{\tau_{\epsilon}})$ are conjugate by an inner automorphism of $\frakg_{\CC}$,
	\item $\tau_{\epsilon}$ commutes with $\theta$,
	\item $(\frakg^{\theta\tau}, \frakg^{\tau_{\epsilon},\theta\tau}) \in F((\frakg^{\theta\tau}, \frakg^{\tau, \theta\tau}))$.
\end{enumerate}
\end{proposition}

\begin{proof}
i) is proved in \cite{OsSe80}.
We give only the explicit form of the automorphism.
Let $\{\alpha_1, \alpha_2, \cdots, \alpha_r\} \subset \Sigma(\fraka)$ be the set of simple roots of $\Sigma(\fraka)$,
and $\{H_1, H_2, \ldots, H_r\} \subset \fraka$ be the dual basis of $\{\alpha_1, \alpha_2, \cdots, \alpha_r\}$.
Then, the automorphism is given by
\begin{align*}
\exp\left(\sum_{i=1}^{r}\frac{\pi \sqrt{-1}}{4}(1-\epsilon(\alpha_i))\ad(H_i)\right).
\end{align*}

For the proof of ii), take $\alpha \in \Sigma(\fraka)\cup \{0\}$ and $X \in \frakg(\fraka;\alpha)$.
Then, we have
\begin{align*}
\theta \tau_{\epsilon}(X)&=\theta \epsilon(\alpha)\tau(X) \\
&=\epsilon(\alpha) \theta\tau(X)\\
&=\epsilon(\alpha) \tau \theta(X).
\end{align*}
Since $\theta(\frakg(\fraka;\alpha)) = \frakg(\fraka;-\alpha)$, we have
\begin{align*}
\epsilon(\alpha) \tau \theta(X) = \tau_{\epsilon}\theta(X).
\end{align*}
This shows ii).

Since $\epsilon(0)=1$, we have $\fraka \subset \frakp^{-\tau,-\tau_\epsilon} \subset \frakg^{\theta\tau}$.
Then, iii) is clear.
\end{proof}

Hereafter, we assume that $\frakg$ is a simple real Lie algebra of Hermitian type,
and $(\frakg, \frakg^{\tau})$ is a holomorphic symmetric pair (see Section \ref{section:holomorphic_symmertic_pair}).
We rewrite an $\epsilon$-family in terms of strongly orthogonal roots.
Fix $\frakt^{\tau}$, $\Delta^{+}$ and strongly orthogonal roots $\{\sor_1, \sor_2, \ldots, \sor_r\}$ in $\frakp^{-\tau}$ as in Section \ref{section:stability_discrete}.
We can take root vectors $X_{\sor_i} \in \frakg^{\tau\theta}_{\CC}(\frakt_{\CC}^{\tau}; \sor_i)$ such that
\begin{align*}
\sor_i([X_{\sor_i}, \overline{X_{\sor_i}}]) = 2.
\end{align*}
Put $\fraka:=\bigoplus_{i=1}^{r} \RR (X_{\sor_i}+\overline{X_{\sor_i}}) \subset \frakp^{-\tau}$ and
$\frakt_0^{\tau}:=\bigoplus_{i=1}^{r} \CC [X_{\sor_i},\overline{X_{\sor_i}}] \subset \frakt^{\tau}_{\CC}$.

We define a \define{Cayley transform} $\mathbf{c}$ by
\begin{align*}
\mathbf{c} := \exp \left(\frac{\pi}{4} \sum_{i=1}^{r}\ad(X_{\sor_i}-\overline{X_{\sor_i}})\right).
\end{align*}
It is known that $\mathbf{c}(\frakt^{\tau}_0)=\fraka_{\CC}$ (see \cite{Kn02}).
We set $\Sigma(\frakt_{0}^{\tau}):= \{\lambda \in (\frakt_{0}^{\tau})^{*} \backslash \{0\}: \frakg_{\CC}(\frakt_{0}^{\tau};\lambda) \neq 0\}$.
Then, $\mathbf{c}$ gives bijections $\mathbf{c}^{*}: \Sigma(\fraka) \rightarrow \Sigma(\frakt_{0}^{\tau}) (\alpha \mapsto \alpha \circ \mathbf{c})$
, and $\mathbf{c}:\frakg_{\CC}(\frakt^{\tau}_{0};\alpha) \rightarrow \frakg_{\CC}(\fraka_{\CC};\alpha \circ \mathbf{c}^{-1})$
for any $\alpha \in \Sigma(\frakt_{0}^{\tau})\cup \{0\}$.

\begin{proposition}\label{proposition:epsilon-family2}
Let $\epsilon$ be a signature such that $(\frakg, \frakg^{\tau_{\epsilon}})$ is a holomorphic symmetric pair.
Then,
\begin{enumerate}[i)]
	\item $\tau\tau_{\epsilon}$ commutes with $\mathbf{c}$,
	\item $\tau\tau_{\epsilon}(X)=\epsilon(\alpha \circ \mathbf{c}^{-1})X$ for any $\alpha \in \Sigma(\frakt^{\tau}_{0})$ and $X \in \frakg_{\CC}(\frakt_{0}^{\tau};\alpha)$,
	\item $\epsilon(\sor_i \circ \mathbf{c}^{-1}) = 1$ for any $i\ (1 \leq i \leq r)$.
\end{enumerate}
\end{proposition}

\begin{proof}
We recall that the characteristic element for $(\frakg, \frakg^{\theta})$ is $\tau$ and $\tau_{\epsilon}$ invariant.
Since $[Z,X_{\sor_i}+\overline{X_{\sor_i}}]=X_{\sor_i}-\overline{X_{\sor_i}}$, we have
\begin{align*}
\tau\tau_{\epsilon}(X_{\sor_i}-\overline{X_{\sor_i}})&=\tau\tau_{\epsilon}([Z,X_{\sor_i}+\overline{X_{\sor_i}}]) \\
&= [Z,\tau\tau_{\epsilon}(X_{\sor_i}+\overline{X_{\sor_i}})] \\
&= [Z,X_{\sor_i}+\overline{X_{\sor_i}}] \\
&= X_{\sor_i}-\overline{X_{\sor_i}}.
\end{align*}
Then, $X_{\sor_i}-\overline{X_{\sor_i}}$ is $\tau\tau_{\epsilon}$-invariant.
This implies that $\mathbf{c}$ commutes with $\tau\tau_{\epsilon}$.
Thus, (i) and (ii) are proved.
Since $X_{\sor_i}+\overline{X_{\sor_i}}$ and $X_{\sor_i}-\overline{X_{\sor_i}}$ are $\tau\tau_{\epsilon}$-invariant,
$X_{\sor_i}$ is $\tau\tau_{\epsilon}$-invariant.
By ii), we have $\epsilon(\sor_i \circ \mathbf{c}^{-1}) = 1$.
\end{proof}

\subsection{Upper bounds and $\epsilon$-family}\label{section:upper_bound_for_hol}

Let $G$ be a connected simple real Lie group of Hermitian type with finite center,
and $\tau$ be an involutive automorphism of $G$ such that $(\frakg, \frakg^{\tau})$ is holomorphic symmetric pair.
We fix a maximal abelian subspace $\fraka$ of $\frakp^{-\tau}$,
and fix a signature $\epsilon$ of $\Sigma(\fraka)$ such that $(\frakg, \frakg^{\tau_{\epsilon}})$ is holomorphic symmetric pair.
Suppose $H$ and $H'$ are analytic subgroups with Lie algebra $\frakg^{\tau}$ and $\frakg^{\tau_{\epsilon}}$.
Applying Theorem \ref{corollary:upper_bound_holomorphic_discrete_series}, we have the following theorem.

\begin{theorem}\label{theorem:upper_bound_general}
Let $\calH$ be a holomorphic discrete series representation of $G$.
Then, we have $C_{H}(\calH)=C_{H'}(\calH)$.
\end{theorem}

\begin{proof}
From Theorem \ref{theorem:stable_theorem_for_holomorphic_discrete_series},
it suffices to show that `$L$'s for $H$ and $H'$ in Section \ref{section:stability_discrete} are conjugate under $K$-action.
We take a Cartan subalgebra $\frakt^{\tau}$ of $\frakk^{\tau}$,
positive system $\Delta^{+}(\frakg^{\tau\theta}_{\CC}, \frakt^{\tau}_{\CC})$
and strongly orthogonal roots $\{\sor_1, \sor_2, \ldots, \sor_r\}$ as in Section \ref{section:stability_discrete}.
We put
\begin{align*}
\fraka' &:=\bigoplus_{i=1}^{r} \RR (X_{\sor_i}+\overline{X_{\sor_i}}) \subset \frakp^{-\tau}, \\
\frakt_0 &:=\bigoplus_{i=1}^{r} \RR \sqrt{-1} [X_{\sor_i},\overline{X_{\sor_i}}] \subset \frakt^{\tau}, \\
L &:= Z_{K\cap H}(\fraka').
\end{align*}

Let us show that $L$ and $Z_{K \cap H \cap H'}(\fraka)$ is conjugate under $K \cap H$-action.
Since $\fraka$ and $\fraka'$ are maximal abelian subspaces of $\frakp^{-\tau}$,
there exists a $k \in K \cap H$ such that $\Ad(k)\fraka'=\fraka$.
Then, We have $\Ad(k)L=Z_{K\cap H}(\fraka)$.
By definition of $\tau_{\epsilon}$ ($\epsilon(0)=1$), we have
\begin{align*}
\Ad(k)\frakl &= Z_{\frakk^{\tau}}(\fraka)\\
&= Z_{\frakk^{\tau,\tau_{\epsilon}}}(\fraka).
\end{align*}
As we stated in the proof of Lemma \ref{lemma:L_is_L_subgroup},
the connected components of $L$ can be controlled by $Z_{T^{\tau}}(\fraka')$.
Here, we denote by $T^{\tau}$ a maximal torus of $K\cap H$ corresponding to $\frakt^{\tau}$.
Since $\fraka$ is contained in $\frakp^{-\tau,-\tau_{\epsilon}}$, we have
\begin{align*}
\tau(\Ad(k)X_{\sor_i})=\tau_{\epsilon}(\Ad(k)X_{\sor_i})=-\Ad(k)X_{\sor_i}.
\end{align*}
This implies that $\Ad(k)\frakt_0$ is contained in $\frakk^{\tau, \tau_{\epsilon}}$.
Moreover, since $\frakt^{\tau} = \frakt_0 \oplus Z_{\frakt^{\tau}}(\fraka')$,
$\Ad(k)\frakt^{\tau}$ is contained in $\frakk^{\tau, \tau_{\epsilon}}$.
Then, $\Ad(k)T^{\tau}$ is contained in $K \cap H \cap H'$.
Therefore, $\Ad(k)L$ is contained in $Z_{K \cap H \cap H'}(\fraka)$.
Since converse inclusion is obvious, we have $\Ad(k)L=Z_{K \cap H \cap H'}(\fraka)$.

Replacing $\tau$ by $\tau_{\epsilon}$, we can show the same argument for $H'$.
This completes the proof.
\end{proof}

\subsection{$\frakk_{\epsilon}$-family case}
We can describe the similarity of branching laws of $\calH|_{H}$ and $\calH|_{H'}$ more precisely.
In this section, we assume that $H$ is a maximal compact subgroup of $G$.

We consider the case that $\tau = \theta$, and fix a signature $\epsilon$ of $\Sigma(\fraka)$ such that $(\frakg, \frakg^{\theta_{\epsilon}})$ is a holomorphic symmetric pair.
In this case, we write $\frakt:=\frakt^{\theta}$ and $\frakt_0:=\frakt^{\theta}_0$ since $\frakt^{\theta}$ is a Cartan subalgebra of $\frakk$.
Note that $\frakt$ is contained in $\frakk^{\theta_{\epsilon}}$ since $\mathbf{c}(\frakt)$ commutes with $\fraka_{\CC}=\mathbf{c}(\frakt_0)$.
Then, $\Delta(\frakk_{\CC}^{\theta_{\epsilon}}, \frakt_{\CC}) \cap \Delta^{+}$ is a positive system
of $\Delta(\frakk_{\CC}^{\theta_{\epsilon}}, \frakt_{\CC})$.

By Proposition \ref{proposition:restricted_root}, the restricted root system $\Sigma(\fraka)$ is type $BC_r$ or $C_r$.
We divide strongly orthogonal roots $\{\sor_1, \sor_2, \ldots, \sor_r\}$ into two parts $\Gamma_1$ and $\Gamma_2$ as follows.
If $\Sigma(\fraka)$ is type $BC_r$, we put
\begin{align*}
\Gamma_1 &:= \{\sor_i: \epsilon(\sor_i / 2) = -1\}, \\
\Gamma_2 &:= \{\sor_i: \epsilon(\sor_i / 2) = 1\}.
\end{align*}
If $\Sigma(\fraka)$ is type $C_r$, we put
\begin{align*}
\Gamma_1 &:= \{\sor_i: \epsilon((\sor_1+\sor_i) / 2) = 1\}, \\
\Gamma_2 &:= \{\sor_i: \epsilon((\sor_1+\sor_i) / 2) = -1\}.
\end{align*}
Here, we identify $\Sigma(\fraka)$ and $\Sigma(\frakt_{0})$ by $\mathbf{c}^{*}$.
If necessary, replacing the positive system,
we can assume
\begin{align}
\begin{split}
\Gamma_1 &= \{\sor_1, \sor_2, \ldots, \sor_{r_1}\}, \\
\Gamma_2 &= \{\sor_{r_1+1}, \sor_{r_1+2}, \ldots, \sor_{r}\}.
\end{split}\label{equation:good_order}
\end{align}

\begin{lemma}\label{lemma:Z'}
We put
\begin{align}\label{equation:characteristic}
Z':= -\sum_{i=1}^{r_1}[X_{\sor_i}, \overline{X_{\sor_i}}].
\end{align}
Then, we have
\begin{align*}
\Delta^{+}(\frakk_{\CC}^{-\theta_{\epsilon}}, \frakt_{\CC})=\{\alpha \in \Delta(\frakk_{\CC}, \frakt_{\CC}): \alpha(Z')=1\}, \\
\Delta(\frakk_{\CC}^{\theta_{\epsilon}}, \frakt_{\CC})=\{\alpha \in \Delta(\frakk_{\CC}, \frakt_{\CC}): \alpha(Z')=0\}.
\end{align*}
Especially, $(\frakk_{\CC}^{-\theta_{\epsilon}})_{+}:=\{X \in \frakk_{\CC}: [Z', X] = X\}$ is stable under $\frakk_{\CC}^{\theta_{\epsilon}}$-action.
\end{lemma}

\begin{proof}
By the definition of $\theta_{\epsilon}$, we have
\begin{align*}
\Delta^{+}(\frakk_{\CC}^{-\theta_{\epsilon}}, \frakt_{\CC})&=\{\alpha \in \Delta^{+}(\frakk_{\CC}, \frakt_{\CC}): \epsilon(\alpha|_{\frakt_0})=-1\}, \\
\Delta(\frakk_{\CC}^{\theta_{\epsilon}}, \frakt_{\CC})&=\{\alpha \in \Delta(\frakk_{\CC}, \frakt_{\CC}): \epsilon(\alpha|_{\frakt_0}) = 1\}.
\end{align*}
For any $\alpha \in \Delta(\frakk_{\CC}, \frakt_{\CC})$, the value $\alpha(Z')$ is determined by the restriction to $\frakt_0$.
Explicit forms of the restrictions of the roots to $\frakt_0$ is given by Proposition \ref{proposition:restricted_root}:
\begin{align*}
\Delta^{+}(\frakk_{\CC}, \frakt_{\CC})&=\bigcup_{1 \leq i<j \leq r}C_{ij} \cup \bigcup_{1 \leq i \leq r}C_{i} \cup C_0, \\
C_{ij} &:= \left\{ \gamma \in \Delta^{+}(\frakk_{\CC}, \frakt_{\CC}): \gamma|_{\frakt_0} = \left.\left(\frac{\gamma_j - \gamma_i}{2}\right) \right|_{\frakt_0} \right\}, \\
C_{i} &:= \left\{ \gamma \in \Delta^{+}(\frakk_{\CC}, \frakt_{\CC}): \gamma|_{\frakt_0} = -\left.\left(\frac{\gamma_i}{2}\right) \right|_{\frakt_0} \right\}, \\
C_{0} &:= \{ \gamma \in \Delta^{+}(\frakk_{\CC}, \frakt_{\CC}): \gamma|_{\frakt_0} = 0 \}.
\end{align*}
From the assumption (\ref{equation:good_order}), we have
\begin{align*}
\epsilon \left(\frac{\sor_j- \sor_i}{2}\right)&=
\begin{cases}
1 & (1 \leq i < j \leq r_1 \text{ or } r_1 < i < j \leq r) \\
-1 & (1 \leq i \leq r_1 < j \leq r)
\end{cases}
, \\
\epsilon \left(-\frac{\sor_i}{2}\right)&=
\begin{cases}
1 & (r_1 < i \leq r) \\
-1 & (1 \leq i \leq r_1)
\end{cases}
.
\end{align*}
Recall that we took $X_{\sor_i}$ satisfying $\sor_i([X_{\sor_i}, \overline{X_{\sor_i}}]) = 2$.
Then, we have
\begin{align*}
\frac{\sor_j- \sor_i}{2}(Z')&=
\begin{cases}
0 & (1 \leq i < j \leq r_1 \text{ or } r_1 < i < j \leq r) \\
1 & (1 \leq i \leq r_1 < j \leq r)
\end{cases}
, \\
-\frac{\sor_i}{2}(Z')&=
\begin{cases}
0 & (r_1 < i \leq r) \\
1 & (1 \leq i \leq r_1)
\end{cases}
.
\end{align*}
This shows the lemma.
\end{proof}

We show the following key lemma.
This lemma can be considered as the scalar type case of the main theorem.

\begin{lemma}\label{lemma:good_ordering}
In the above settings, we have the following equations:
\begin{align*}
\dweight_{\frakk^{\theta_{\epsilon}}}(\frakp_{+}^{-\theta_{\epsilon}}) &=
\left\{-\sum_{i=1}^{r}c_i \sor_i: c_1\geq c_2 \geq \cdots c_{r_1} \geq 0, c_{r_1 + 1}\geq c_{r_1+2} \geq \cdots c_{r} \geq 0\right\}, \\
\dweight_{\frakk}(\frakp_{+})
&=\{\lambda \in \dweight_{\frakk^{\theta_{\epsilon}}}(\frakp_{+}^{-\theta_{\epsilon}}):
(\lambda, \alpha) \geq 0 \text{ for any } \alpha \in \Delta^{+}(\frakk_{\CC}^{-\theta_{\epsilon}}, \frakt_{\CC})\}.
\end{align*}
\end{lemma}

\begin{proof}
For the first equation, we show that $\frakg^{\theta\theta_{\epsilon}}$ has at most two non-compact simple factors
determined by $\Gamma_1$ and $\Gamma_2$.
We define two subsets $\Sigma_1$ and $\Sigma_2$ of $\Sigma(\fraka)$ as follows:
\begin{align*}
\Sigma_i := \left\{\pm \frac{\sor_i \pm \sor_j}{2}: \sor_i, \sor_j \in \Gamma_i\right\} \cup
\left\{\pm \frac{\sor_i}{2} \in \Sigma(\fraka): \sor_i \in \Gamma_i, \epsilon \left(\frac{\sor_i}{2} \right)=1 \right\} \cup \Gamma_i.
\end{align*}
By definition of $\Gamma_i$ and Proposition \ref{proposition:epsilon-family2}, $\Sigma_i$ is a subroot system of $\Sigma(\fraka)$,
and if $\alpha \in \Sigma(\fraka)$ satisfies $\epsilon(\alpha)=1$, $\alpha$ is an element of either of $\Sigma_1$ or $\Sigma_2$.
This implies that $\Sigma_1 \cup \Sigma_2$ is the restricted root system of $\frakg^{\theta\theta_{\epsilon}}$ with respect to $\fraka$.
Since $\Sigma_1$ and $\Sigma_2$ are irreducible root systems, $\frakg^{\theta\theta_{\epsilon}}$ has at most two non-compact simple factors.
By Kostant-Hua-Schmid Theorem (Proposition \ref{proposition:kostant_hua_schmid}), the first equation is proved.

By Kostant-Hua-Schmid Theorem again, the left hand side of the second equation is contained in the right hand side.
If $r_1=r$, we have nothing to prove.
Then, we assume $r_1 < r$.
We take $-\sum_{i=1}^{r}c_i \sor_i \in \dweight_{\frakk^{\theta_{\epsilon}}}(\frakp_{+}^{-\theta_{\epsilon}})$
such that $(-\sum_{i=1}^{r}c_i \sor_i, \alpha) \geq 0$ for any $\alpha \in \Delta^{+}(\frakk_{\CC}^{-\theta_{\epsilon}}, \frakt_{\CC})$.
Especially, we can choose $\alpha$ such that $\alpha|_{\frakt_0} = (\sor_{r_1+1}-\sor_{r_1})/2$.
Then, we have
\begin{align*}
0 &\leq \left(-\sum_{i=1}^{r}c_i \sor_i, \alpha\right)\\
&= \left(-\sum_{i=1}^{r}c_i \sor_i, \frac{\sor_{r_1+1}-\sor_{r_1}}{2}\right)\\
&= (\sor_i, \sor_i)\frac{c_{r_1}-c_{r_1+1}}{2}.
\end{align*}
Therefore, $c_{r_1}\geq c_{r_1+1}$. This shows the second equation.
\end{proof}

Let $H$ be the analytic subgroup with Lie algebra $\frakg^{\theta_{\epsilon}}$.

\begin{theorem}\label{theorem:upper_bound_compact}
Let $\calH$ be a holomorphic discrete series representation of $G$.
Suppose $(\calH^{\frakp_{+}})^{*}$ has the following formal character with respect to $\frakt$:
\begin{align*}
\ch((\calH^{\frakp_{+}})^{*})=\bigoplus_{\nu \in \frakt^{*}} m(\nu)e^{\nu}.
\end{align*}
We put $\calV:=\{\nu \in \sqrt{-1}\frakt^{*}: m(\nu) \neq 0\}$.
Then, for $\lambda \in \sqrt{-1}\frakt^{*}$ such that $(\lambda + \nu, \alpha) \geq 0$ for any $\alpha \in \Delta^{+}(\frakt_{\CC}, \frakk_{\CC}^{-\theta_{\epsilon}})$ and $\nu \in \calV$,
we have
\begin{align*}
m^{K}_{\calH}(\lambda) = m^{H}_{\calH}(\lambda).
\end{align*}
\end{theorem}

\begin{proof}
By Proposition \ref{proposition:reduction_to_compact}, we consider a branching law of $S(\frakp^{-\theta_{\epsilon}}_{-})\otimes \calH^{\frakp_{+}}|_{K\cap H}$.
From the Weyl character formula, we have
\begin{align}
D_K\ch(V_{\lambda,K}\otimes (\calH^{\frakp_{+}})^{*}) &= D_K\ch(V_{\lambda,K})\ch((\calH^{\frakp_{+}})^{*}) \notag \\
&= \sum_{\sigma \in W_{K}}\sgn(\sigma)e^{\sigma(\lambda+\rho_{K})-\rho_{K}} \sum_{\nu \in \calV}m(\nu)e^{\nu} \notag \\
&= \sum_{\nu \in \calV}m(\nu) \sum_{\sigma \in W_{K}} \sgn(\sigma)e^{\sigma(\lambda+\nu+\rho_{K})-\rho_{K}}. \label{equation:character}
\end{align}
Here, we write the Weyl group for $(K,T)$ by $W_K$, half the sum of the positive roots by $\rho_K$,
and the Weyl denominator for $(K,T)$ by $D_K$.
For the third equality, we used $W_K$-invariance of $m(\nu)$.
We have the similar equations for $K\cap H$ by replacing $W_K$ by $W_{K\cap H}$, and $\rho_K$ by $\rho_{K\cap H}$.

We put
\begin{align*}
I_K(\lambda) &:= \{(\nu, \sigma) \in \calV \times W_K: \sigma(\lambda+\nu+\rho_{K})-\rho_{K} \in \dweight(\frakp_{+})\}, \\
I_{K\cap H}(\lambda) &:= \{(\nu, \sigma) \in \calV \times W_{K\cap H}: \sigma(\lambda+\nu+\rho_{K\cap H})-\rho_{K \cap H} \in \dweight(\frakp_{+}^{-\theta_{\epsilon}})\}.
\end{align*}
By Lemma \ref{lemma:Z'}, we have $\sigma(\rho_{K})-\rho_{K}=\sigma(\rho_{K\cap H})-\rho_{K\cap H}$ for any $\sigma \in W_{K\cap H}$.
Then, we can replace $\rho_{K\cap H}$ by $\rho_{K}$ in the above definitions.
Note that if $(\nu, \sigma_1), (\nu, \sigma_2) \in I_{\cdot}(\lambda)$, then $\sigma_1 = \sigma_2$.

From (\ref{equation:character}), we can write the multiplicity by this notation:
\begin{align*}
m^{K}_{\calH}(\lambda) &= \sum_{(\nu, \sigma) \in I_{K}(\lambda)}\sgn(\sigma) m(\nu), \\
m^{H}_{\calH}(\lambda) &= \sum_{(\nu, \sigma) \in I_{K\cap H}(\lambda)}\sgn(\sigma) m(\nu).
\end{align*}
Then, it suffices to show $I_{K}(\lambda)=I_{K\cap H}(\lambda)$ if $\lambda$ satisfies the condition.

We assume $(\lambda + \nu, \alpha) \geq 0$ for any $\alpha \in \Delta^{+}(\frakk_{\CC}^{-\theta_{\epsilon}}, \frakt_{\CC})$ and $\nu \in \calV$.
First, we will show $I_K(\lambda)\supset I_{K\cap H}(\lambda)$.
Take $(\nu, \sigma) \in I_{K\cap H}(\lambda)$.
By definition, we have $\sigma(\lambda+\nu+\rho_{K})-\rho_{K} \in \dweight(\frakp_{+}^{-\theta_{\epsilon}})$.
From Lemma \ref{lemma:Z'}, $\Delta^{+}(\frakk_{\CC}^{-\theta_{\epsilon}}, \frakt_{\CC})$ is stable under $\sigma$.
Take a lowest root $\alpha \in \Delta^{+}(\frakk_{\CC}^{-\theta_{\epsilon}}, \frakt_{\CC})$
with respect to $\Delta^{+}(\frakk_{\CC}, \frakt_{\CC})$.
Since $\alpha$ is a lowest root, $\sigma^{-1}(\alpha)-\alpha$ is a sum of elements of
$\Delta^{+}(\frakk^{\theta_\epsilon}, \frakt_{\CC})$ with positive coefficients.
Then, we have
\begin{align*}
(\sigma(\lambda+\nu+\rho_{K})-\rho_{K}, \alpha)&=(\sigma(\lambda+\nu+\rho_{K}), \alpha)-(\rho_{K}, \alpha)\\
&=(\lambda+\nu+\rho_{K},\sigma^{-1}(\alpha))-(\rho_{K}, \alpha)\\
&=(\lambda+\nu,\alpha)+(\rho_{K}, \sigma^{-1}(\alpha)-\alpha)\geq 0.
\end{align*}
Thus, $\sigma(\lambda+\nu+\rho_{K})-\rho_{K}$ is dominant with respect to $\Delta^{+}(\frakk_{\CC}, \frakt_{\CC})$.
By Lemma \ref{lemma:good_ordering}, we have $\sigma(\lambda+\nu+\rho_{K})-\rho_{K} \in \dweight(\frakp_{+})$.
This implies the desired inclusion.

Next, we show the converse inclusion.
Take $(\nu, \sigma) \in I_{K}(\lambda)$.
Since $\sigma(\lambda+\nu+\rho_{K})-\rho_{K} \in \dweight(\frakp_{+})$,
$\lambda+\nu+\rho_{K}$ is regular for $W_K$, especially for $W_{K\cap H}$.
Then, there exist a $\sigma' \in W_{K\cap H}$ such that $\sigma'(\lambda+\nu+\rho_{K})$ is
strictly dominant with respect to $\Delta^{+}(\frakt_{\CC}, \frakk_{\CC}^{\theta_{\epsilon}})$.
Since $\Delta^{+}(\frakk^{-\theta_\epsilon}_{\CC}, \frakt_{\CC})$ is $W_{K\cap H}$-invariant,
$\sigma'(\lambda+\nu+\rho_{K})$ is also strictly dominant with respect to $\Delta^{+}(\frakk_{\CC}, \frakt_{\CC})$.
Then, we have $\sigma = \sigma'$.
This completes the proof.
\end{proof}

\begin{remark}
Note that the irreducibility of $\calH$ is not used in the above proof.
Then, the assumption that $\calH$ is a holomorphic discrete series representation can be replaced by the assumption
that $\calH$ is the finite direct sum of holomorphic discrete series representations.
\end{remark}

\subsection{Non-compact case}\label{section:non-compact_upper_bound}

The aim of this section is to show an analogous theorem of Theorem \ref{theorem:upper_bound_compact} for non-compact symmetric pairs.
We use the notation in Section \ref{section:upper_bound_for_hol}.

We can apply two types of transformations to the branching law of $\calH|_H$:
\begin{enumerate}[i)]
	\item the branching law of $\calH|_H$ coincides with the branching law of $(\calU(\frakg^{\theta\tau})\calH^{\frakp_{+}})|_{\frakg^{\theta,\tau}}$
	in the sense of Proposition \ref{proposition:reduction_to_compact}, and
	\item if $H$ is a maximal compact subgroup of $G$, the branching law of $\calH|_H$ coincides with the branching law of $\calH|_{H'}$
	in the sense of Theorem \ref{theorem:upper_bound_compact}.
\end{enumerate}
Then, if we ignore the assumption (\ref{equation:good_order}) in Theorem \ref{theorem:upper_bound_compact}, 
the branching laws of $\calH|_H$ and $\calH|_{H'}$ are reduced to the same branching law as follows.

\begin{align}
\xymatrix{
(\frakg, \frakg^\tau) \ar@{-->}[dd]^{\epsilon \text{-family}} \ar[r]^(0.4){\text{i)}} & (\frakg^{\theta\tau}, \frakg^{\theta\tau, \tau}) \ar[r]^{\text{ii)}}
& (\frakg^{\theta\tau}, \frakg^{\theta\tau, \tau_{\epsilon}})\ar[rd]^{\text{i)}} \\
&&&(\frakg^{\theta\tau, \theta\tau_{\epsilon}}, \frakg^{\theta, \tau, \tau_{\epsilon}})\\
(\frakg, \frakg^{\tau_\epsilon}) \ar[r]^(0.4){\text{i)}} & (\frakg^{\theta\tau_{\epsilon}}, \frakg^{\theta\tau_{\epsilon}, \tau_\epsilon}) \ar[r]^{\text{ii)}}
& (\frakg^{\theta\tau_{\epsilon}}, \frakg^{\theta\tau_\epsilon, \tau}) \ar[ur]_{\text{i)}} 
}\label{diagram:reduction}
\end{align}

Let us show the existence of a compatible ordering of $(\frakt^{\tau})^{*}$ with respect to the pairs
$(\frakg^{\theta\tau}, \frakg^{\theta\tau, \tau_{\epsilon}})$ and $(\frakg^{\theta\tau_{\epsilon}}, \frakg^{\theta\tau_\epsilon, \tau})$
satisfying (\ref{equation:good_order}).

\begin{lemma}\label{lemma:good_ordering_for_non_compact}
In the above settings, there exists an ordering of $\sqrt{-1}(\frakt^{\tau})^{*}$ that satisfies the following conditions:
\begin{enumerate}[i)]
	\item any elements of $\Delta(\frakp^{-\tau}_{+}, \frakt^{\tau}_{\CC})$ and $\Delta(\frakp^{-\tau_{\epsilon}}_{+}, \frakt^{\tau}_{\CC})$ are positive,
	\item $\dweight_{\frakk^{\tau}}(\frakp^{-\tau}_{+})=\{\lambda \in \dweight_{\frakk^{\tau, \tau_{\epsilon}}}(\frakp_{+}^{-\tau, -\tau_{\epsilon}}):
(\lambda, \alpha) \geq 0 \text{ for any } \alpha \in \Delta^{+}(\frakk_{\CC}^{\tau, -\tau_{\epsilon}}, \frakt^{\tau}_{\CC})\}$,
	\item $\dweight_{\frakk^{\tau_{\epsilon}}}(\frakp^{-\tau_\epsilon}_{+})=\{\lambda \in \dweight_{\frakk^{\tau, \tau_{\epsilon}}}(\frakp_{+}^{-\tau, -\tau_{\epsilon}}):
(\lambda, \alpha) \geq 0 \text{ for any } \alpha \in \Delta^{+}(\frakk_{\CC}^{-\tau, \tau_{\epsilon}}, \frakt^{\tau}_{\CC})\}$.
\end{enumerate}
\end{lemma}

\begin{proof}
We put a lexicographical order on $\sqrt{-1}(\frakt^{\tau})^{*}$ as follows.

Let $\{\fraks_1, \fraks_2, \ldots, \fraks_n\}$ be the set of the non-compact simple factors of $\frakg^{\theta\tau, \theta\tau_{\epsilon}}$.
We can assume that if $\fraks_i$ is of tube type and $\fraks_j$ is of non-tube type, then $i < j$.
We reorder $\{X_1, X_2, \ldots, X_r\}$ (see Section \ref{section:epsilon_family})
such that if $X_i \in \fraks_{n_i}$, $X_j \in \fraks_{n_j}$ and $n_i < n_j$, then $i < j$.
We also reorder $\{\sor_1, \sor_2, \ldots, \sor_r\}$ according to the order of $\{X_1, X_2, \ldots, X_r\}$.
Based on this ordering, we take an ordered generating set of $\sqrt{-1}\frakt^{\tau}$ as
\begin{align*}
\calB := \{Z, -[X_1, \overline{X_1}], -[X_{2}, \overline{X_{2}}], \ldots, -[X_r, \overline{X_r}], Y_1, Y_2, \ldots, Y_d\}.
\end{align*}
Here, $Z$ is a characteristic element of $\frakg$, and $\{Y_1, Y_2, \ldots, Y_d\}$ is a basis of $(\frakt^{\tau}_0+\CC Z)^{\perp}\cap \sqrt{-1}\frakt^{\tau}$.
$\calB$ is a generating set of $\sqrt{-1}\frakt^{\tau}$ although it may not be a basis.
Then, $\calB$ induces a lexicographical ordering on $\sqrt{-1}(\frakt^{\tau})^{*}$.

Let us show that the ordering satisfies the desired conditions.
The condition i) is clear from the definition of $\calB$.
We take a maximal set of strongly orthogonal roots $\{\sor_1', \sor_2', \ldots, \sor_r'\}$ from $\Delta^+(\frakp_+^{-\tau}, \frakt_\CC^\tau)$ with respect to the new ordering.
By Proposition \ref{proposition:restricted_root}, we have $\sor_i' = \sor_i$ for any $1 \leq i \leq r$.
For each non-compact simple factor of $\frakg^{\tau\theta}$,
the assumption (\ref{equation:good_order}) of Lemma \ref{lemma:Z'} is satisfied.
Then, the condition ii) is proved from Lemma \ref{lemma:good_ordering}.
Since the condition iii) can be proved in the same way as ii), then Lemma \ref{lemma:good_ordering_for_non_compact} is proved.
\end{proof}

By Lemma \ref{lemma:good_ordering_for_non_compact} and the reduction (\ref{diagram:reduction}),
we have the following theorem.

\begin{theorem}\label{theorem:upper_bound_hol}
Let $\calH$ be a holomorphic discrete series representation of $G$.
We put an ordering on $\sqrt{-1}\frakt^{\tau}$ as in Lemma \ref{lemma:good_ordering_for_non_compact}.
Suppose $(\calH^{\frakp_{+}})^{*}$ has the following formal character with respect to $\frakt^{\tau}$:
\begin{align*}
\ch((\calH^{\frakp_{+}})^{*})=\bigoplus_{\nu \in (\frakt^{\tau})^{*}} m(\nu)e^{\nu}.
\end{align*}
We put $\calV:=\{\nu \in \sqrt{-1}(\frakt^{\tau})^{*}: m(\nu) \neq 0\}$.
Then, for $\lambda \in \sqrt{-1}(\frakt^{\tau})^{*}$ such that $(\lambda + \nu, \alpha) \geq 0$ for any
$\alpha \in \Delta^{+}(\frakk_{\CC}^{-\tau\tau_{\epsilon}}, \frakt^{\tau}_{\CC})$ and $\nu \in \calV$,
we have
\begin{align*}
m^{H}_{\calH}(\lambda) = m^{H'}_{\calH}(\lambda).
\end{align*}
\end{theorem}

\begin{table}[h]
\centering
\caption{holomorphic symmetric pairs and $\epsilon$-family}
\begin{tabular}{|c|c|c|c|}
\hline
$\frakg$& parameter &$\{\frakg^{\tau_{\epsilon}}: (\frakg, \frakg^{\tau_{\epsilon}}) \text{ is holomorphic}\}$ \\ \hline \hline
$\su(p,q)$& $i,j$ &$\{\fraks(\uor(i-k,j+k)+\uor(p-i+k,q-j-k))\}$\\ \hline
$\su(n,n)$& & $\{\so^{*}(2n)\}$ \\ \hline
$\su(n,n)$& & $\{\spor(n,\RR)\}$ \\ \hline
$\so^{*}(2n)$& & $\{\uor(2i,n-2i):0\leq 2i\leq n\}$ \\ \hline
$\so^{*}(2n)$& & $\{\uor(2i+1,n-2i-1):0\leq 2i+1\leq n\}$ \\ \hline
$\so^{*}(2n)$& $i$ & $\{\so^{*}(2i)+\so^{*}(2(n-i))\}$\\ \hline
$\so(2,n)$& $i$ & $\{\so(2,n-i)+\so(i), \so(n-i+2)+\so(2,i-2)\}$ \\ \hline
$\so(2,2n)$& & $\{\uor(1,n)\}$ \\ \hline
$\spor(n,\RR)$& & $\{\uor(i,n-i):0\leq i \leq n\}$ \\ \hline
$\spor(n,\RR)$& $i$ & $\{\spor(i,\RR)+\spor(n-i,\RR)\}$ \\ \hline
$\frake_{6(-14)}$& & $\{\so(10)+\so(2), \so^{*}(10)+\so(2), \so(8,2)+\so(2)\}$ \\ \hline
$\frake_{6(-14)}$& & $\{\su(5,1)+\slor(2,\RR), \su(4,2)+\su(2)\}$ \\ \hline
$\frake_{7(-25)}$& & $\{\frake_{6(-78)}+\so(2), \frake_{6(-14)}+\so(2)\}$ \\ \hline
$\frake_{7(-25)}$& & $\{\so(10,2)+\slor(2,\RR), \so^{*}(12)+\su(2)\}$ \\ \hline
$\frake_{7(-25)}$& & $\{\su(6,2)\}$ \\ \hline
\end{tabular}
\end{table}
\newpage


\bibliography{reference.bib}

\def\cprime{$'$}
\begin{thebibliography}{10}

\bibitem{BLV86}
M.~Brion, D.~Luna, and Th. Vust.
\newblock Espaces homog\`enes sph\'eriques.
\newblock {\em Invent. Math.}, 84(3):617--632, 1986.

\bibitem{Gr83}
F.~D. Grosshans.
\newblock The invariants of unipotent radicals of parabolic subgroups.
\newblock {\em Invent. Math.}, 73(1):1--9, 1983.

\bibitem{Gr92}
Frank~D. Grosshans.
\newblock Contractions of the actions of reductive algebraic groups in
  arbitrary characteristic.
\newblock {\em Invent. Math.}, 107(1):127--133, 1992.

\bibitem{He94}
Sigurdur Helgason.
\newblock {\em Geometric analysis on symmetric spaces}, volume~39 of {\em
  Mathematical Surveys and Monographs}.
\newblock American Mathematical Society, Providence, RI, 1994.

\bibitem{Ho89}
Roger Howe.
\newblock Remarks on classical invariant theory.
\newblock {\em Trans. Amer. Math. Soc.}, 313(2):539--570, 1989.

\bibitem{Hu63}
L.~K. Hua.
\newblock {\em Harmonic analysis of functions of several complex variables in
  the classical domains}.
\newblock Translated from the Russian by Leo Ebner and Adam Kor\'anyi. American
  Mathematical Society, Providence, R.I., 1963.

\bibitem{Jo92}
Anthony Joseph.
\newblock Annihilators and associated varieties of unitary highest weight
  modules.
\newblock {\em Ann. Sci. \'Ecole Norm. Sup. (4)}, 25(1):1--45, 1992.

\bibitem{Kn02}
Anthony~W. Knapp.
\newblock {\em Lie groups beyond an introduction}, volume 140 of {\em Progress
  in Mathematics}.
\newblock Birkh\"auser Boston Inc., Boston, MA, second edition, 2002.

\bibitem{Ko94}
Toshiyuki Kobayashi.
\newblock Discrete decomposability of the restriction of
  {$A_{\frakq}(\lambda)$} with respect to reductive subgroups and its
  applications.
\newblock {\em Invent. Math.}, 117(2):181--205, 1994.

\bibitem{Ko98}
Toshiyuki Kobayashi.
\newblock Discrete series representations for the orbit spaces arising from two
  involutions of real reductive {L}ie groups.
\newblock {\em J. Funct. Anal.}, 152(1):100--135, 1998.

\bibitem{Ko04}
Toshiyuki Kobayashi.
\newblock Geometry of multiplicity-free representations of {${\rm GL}(n)$},
  visible actions on flag varieties, and triunity.
\newblock {\em Acta Appl. Math.}, 81(1-3):129--146, 2004.

\bibitem{Ko05}
Toshiyuki Kobayashi.
\newblock Multiplicity-free representations and visible actions on complex
  manifolds.
\newblock {\em Publ. Res. Inst. Math. Sci.}, 41(3):497--549, 2005.

\bibitem{Ko07_2}
Toshiyuki Kobayashi.
\newblock A generalized {C}artan decomposition for the double coset space
  {$({\rm U}(n_1)\times{\rm U}(n_2)\times{\rm U}(n_3))\backslash{\rm
  U}(n)/({\rm U}(p)\times{\rm U}(q))$}.
\newblock {\em J. Math. Soc. Japan}, 59(3):669--691, 2007.

\bibitem{Ko08}
Toshiyuki Kobayashi.
\newblock Multiplicity-free theorems of the restrictions of unitary highest
  weight modules with respect to reductive symmetric pairs.
\newblock In {\em Representation theory and automorphic forms}, volume 255 of
  {\em Progr. Math.}, pages 45--109. Birkh\"auser Boston, Boston, MA, 2008.

\bibitem{Ko13}
Toshiyuki Kobayashi.
\newblock Propagation of the multiplicity-freeness property for holomorphic
  vector bundles.
\newblock In {\em Lie Groups: Structure, Actions, and Representations}, volume
  306 of {\em Progr. Math.} Birkh\"auser Basel, 2013.

\bibitem{Li99}
Jian-Shu Li.
\newblock The correspondences of infinitesimal characters for reductive dual
  pairs in simple {L}ie groups.
\newblock {\em Duke Math. J.}, 97(2):347--377, 1999.

\bibitem{Mo08}
Masayasu Moriwaki.
\newblock Multiplicity-free decompositions of the minimal representation of the
  indefinite orthogonal group.
\newblock {\em Internat. J. Math.}, 19(10):1187--1201, 2008.

\bibitem{OsSe80}
Toshio {\=O}shima and Jir{\=o} Sekiguchi.
\newblock Eigenspaces of invariant differential operators on an affine
  symmetric space.
\newblock {\em Invent. Math.}, 57(1):1--81, 1980.

\bibitem{OsSe84}
Toshio {\=O}shima and Jir{\=o} Sekiguchi.
\newblock The restricted root system of a semisimple symmetric pair.
\newblock In {\em Group representations and systems of differential equations
  ({T}okyo, 1982)}, volume~4 of {\em Adv. Stud. Pure Math.}, pages 433--497.
  North-Holland, Amsterdam, 1984.

\bibitem{PeSe12}
Ivan Penkov and Vera Serganova.
\newblock On bounded generalized {H}arish-{C}handra modules.
\newblock {\em Ann. Inst. Fourier (Grenoble)}, 62(2):477--496, 2012.

\bibitem{Pr96}
Tomasz Przebinda.
\newblock The duality correspondence of infinitesimal characters.
\newblock {\em Colloq. Math.}, 70(1):93--102, 1996.

\bibitem{Ro79}
W.~Rossmann.
\newblock The structure of semisimple symmetric spaces.
\newblock {\em Canad. J. Math.}, 31(1):157--180, 1979.

\bibitem{Sa09}
Atsumu Sasaki.
\newblock Visible actions on irreducible multiplicity-free spaces.
\newblock {\em Int. Math. Res. Not. IMRN}, (18):3445--3466, 2009.

\bibitem{Sa93}
Fumihiro Sat{\=o}.
\newblock On the stability of branching coefficients of rational
  representations of reductive groups.
\newblock {\em Comment. Math. Univ. St. Paul.}, 42(2):189--207, 1993.

\bibitem{Sh69}
Wilfried Schmid.
\newblock Die {R}andwerte holomorpher {F}unktionen auf hermitesch symmetrischen
  {R}\"aumen.
\newblock {\em Invent. Math.}, 9:61--80, 1969/1970.

\bibitem{Ta12}
Yuichiro Tanaka.
\newblock Classification of visible actions on flag varieties.
\newblock {\em Proc. Japan Acad. Ser. A Math. Sci.}, 88(6):91--96, 2012.

\bibitem{Ti11}
Dmitry~A. Timashev.
\newblock {\em Homogeneous spaces and equivariant embeddings}, volume 138 of
  {\em Encyclopaedia of Mathematical Sciences}.
\newblock Springer, Heidelberg, 2011.
\newblock Invariant Theory and Algebraic Transformation Groups, 8.

\bibitem{Vo89}
David~A. Vogan, Jr.
\newblock Associated varieties and unipotent representations.
\newblock In {\em Harmonic analysis on reductive groups ({B}runswick, {ME},
  1989)}, volume 101 of {\em Progr. Math.}, pages 315--388. Birkh\"auser
  Boston, Boston, MA, 1991.

\bibitem{Ya05}
Hiroshi Yamashita.
\newblock Isotropy representation for {H}arish-{C}handra modules.
\newblock In {\em Infinite dimensional harmonic analysis {III}}, pages
  325--351. World Sci. Publ., Hackensack, NJ, 2005.

\end{thebibliography}

\end{document}